\documentclass{amsart}
\usepackage{etex}
\usepackage{lineno}
\usepackage[arrow,matrix,graph,frame,poly,arc,tips]{xy}
\usepackage{amscd,amssymb,subfigure,epsfig,color}
\usepackage[width=5.65in,height=8.0in,centering]{geometry}
\usepackage{amsmath}
\usepackage{fancyvrb,verbatim}
\usepackage[colorlinks,plainpages,backref]{hyperref}
\usepackage{tikz}
\usetikzlibrary{decorations,calc}
\usepackage{booktabs}

\theoremstyle{plain}
\newtheorem{theorem}{Theorem}[section]
\newtheorem{lemma}[theorem]{Lemma}
\newtheorem{proposition}[theorem]{Proposition}

\theoremstyle{definition}
\newtheorem{rem}[theorem]{Remark}
\newtheorem{definition}[theorem]{Definition}
\newtheorem{exm}[theorem]{Example}
\newtheorem{exr}[theorem]{Exercise}
\newtheorem*{ack}{Acknowledgement}

\newenvironment{example}%
{\pushQED{\qed}\begin{exm}}%
{\popQED\end{exm}}  % end example environment with a qed-symbol
\newenvironment{exercise}%
{\pushQED{\qed}\begin{exr}}%
{\popQED\end{exr}}  % end exercise environment with a qed-symbol
\newenvironment{remark}%
{\pushQED{\qed}\begin{rem}}%
{\popQED\end{rem}}  % end with a qed-symbol
\numberwithin{equation}{section}

\newcommand{\abs}[1]{{\lvert #1 \rvert}}

\newcommand{\set}[1]{\left\{#1\right\}}

\DeclareMathOperator{\M}{\mathsf M}

\newcommand{\A}{{\ensuremath{\mathcal A}}}
\newcommand{\B}{{\mathcal B}}   % bases of a matroid
\newcommand{\CC}{{\mathcal C}}  % coordinate hyperplanes

\newcommand{\F}{{\mathcal F}}

\newcommand{\X}{{\mathfrak X}}

\newcommand{\Z}{{\mathbb Z}}

\newcommand{\R}{{\mathbb R}}
\newcommand{\C}{{\mathbb C}}

\renewcommand{\k}{{\Bbbk}}
\newcommand{\irr}{{\mathrm{irr}}}  % irreducible

\newcommand{\bottom}{{\hat{0}}}
\renewcommand{\top}{[n]}

\newcommand{\m}{{\mathbf m}} % a multiplicity vector?

\renewcommand{\P}{{\mathbb P}}
\newcommand{\PT}{T_0}  % projective torus

\newcommand{\crit}{{\mathcal S}}  % critical set variety

\DeclareMathOperator{\Hom}{Hom}

\DeclareMathOperator{\rank}{rank}

\DeclareMathOperator{\Gr}{Gr}  % Grassmannian?
\DeclareMathOperator{\stab}{stab} % stablilizer

\DeclareMathOperator{\vc}{vc}  % visible contours
\DeclareMathOperator{\dCP}{wnd}  % wonderful model
\newcommand{\Be}{\widetilde{B}}  % Bergman fan
\newcommand{\Ne}{\mathcal{N}}    % nested set complex
\newcommand{\Nfan}{\widetilde{\Ne}}    % nested set fan

\DeclareMathOperator{\Log}{Log}

\DeclareMathOperator{\In}{In}    % initial in term order
\newcommand{\cl}[1]{{\rm cl}(#1)}  % matroid closure
\newcommand{\supp}[1]{{\rm supp}(#1)}
\newcommand{\conv}{{\rm conv}}
\newcommand{\row}{\mathsf{row}}
\newcommand{\col}{\mathsf{col}}

\begin{document}
%%%%%%%%%%%%%%%%%%%%%%%%%%%%%%%%%%%%%%%%%%%%%%%%%%%%%%%%%%%%%%%%%%%%%%%%%%%%%%%

\title{Toric and tropical compactifications of hyperplane complements}

\author{Graham Denham}
\address{
G.~Denham\\
Department of Mathematics\\
University of Western Ontario\\
London, ON N6A 5B7\\
Canada
}
\email{gdenham@uwo.ca}
\thanks{Partially supported by a grant from NSERC of Canada.}

\date{\today}

%%%%%%%%%%%%%%%%%%%%%%%%%%%%%%%%%%%%%%%%%%%%%%%%%%%%%%%%%%%%%%%%%%%%%%%%%%%%%%%

\begin{abstract}
These lecture notes survey and compare various compactifications of
complex hyperplane arrangement complements.  In particular,
we review the Gel$'$fand-MacPherson
construction, Kapranov's visible contours compactification, and De Concini
and Procesi's wonderful compactification.  We explain how these constructions
are unified by some ideas from the modern origins of tropical geometry.
\end{abstract}



\maketitle
\setcounter{tocdepth}{1}
\tableofcontents

%%%%%%%%%%%%%%%%%%%%%%%%%%%%%%%%%%%%%%%%%%%%%%%%%%%%%%%%%%%%%%%%%%%%%%%%%%%%%%%
\section{Introduction}
%%%%%%%%%%%%%%%%%%%%%%%%%%%%%%%%%%%%%%%%%%%%%%%%%%%%%%%%%%%%%%%%%%%%%%%%%%%%%%%

The purpose of this paper is to give a gentle introduction to toric and tropical
compactifications of linear spaces.  The material is based on lectures given
at the graduate student 
summer school ``Arrangements in Pyr\'en\'ees'' in June 2012.    
The presentation gives emphasis to 
explicit examples and calculations, and we suggest some exercises 
for the reader.
The material here can be found in various sources, and we give a guide to 
the literature without attempting to be comprehensive.
Some proofs are included, particularly when it is possible to make a
short or self-contained argument.

We review some terminology of matroid theory
in \S\ref{sec:arrangements}, and note the relationship between matroid
realizations, hyperplane arrangements, and linear subspaces of algebraic
tori.  In \S\ref{sec:toric}, we briefly recall
those aspects of the theory of toric varieties that play a key role in
the constructions we consider.
After some discussion of some basics of tropical geometry
(\S\ref{sec:tropical}), arrangement compactifications appear
in \S\ref{sec:cpct}.  We single out two important developments in 
the field.  The first is Gel$'$fand, Goresky, MacPherson
and Serganova's construction of the 
matroid stratification of the
Grassmannian, \cite{GGMS87,GM82}.  The second is De Concini and Procesi's
wonderful compactification~\cite{deCP95}.  We outline in 
detail some discoveries that tie the two together, due to
Feichtner and Sturmfels~\cite{FS05}, and Ardila and Klivans~\cite{AK06}.

\section{Hyperplane Arrangements}\label{sec:arrangements}
\subsection{Matroids}
A matroid is a structure that abstracts the linear (in)dependence properties
of a finite collection of vectors.  Matroids have numerous equivalent
definitions, and we mention the book \cite{Oxbook} as a detailed reference.
As a brief introduction, let us define a matroid $\M$ to be a nonempty collection
of subsets $\B$ of a finite set $E$ that possesses the following property.
That is, if $B,B'\in\B$, and $x\in B-B'$, then there is an element $y\in B'-B$
for which $(B-\set{x})\cup\set{y}\in\B$.

Then $\B$ is called the set of {\em bases} of the matroid.  One can prove
that all bases have the same cardinality.  This number is called the {\em rank}
of the matroid, an integer $d\geq0$.  Let
\[
\In(\M)=\set{I\subseteq E\colon I\subseteq B\text{~for some~}B\in \B},
\]
called the {\em independent} sets of $\M$.
The reader may notice that this exactly to say that the
$\In(\M)$ forms a simplicial complex with vertices $E$ in which the
bases are the maximal simplices.  This is called the {\em matroid complex}
of $\M$.

Without loss, we will assume $E=[n]:=\set{1,2,\ldots,n}$, for some $n\geq1$.
The {\em rank} of a set $S\subseteq [n]$ is defined to be the cardinality of
its largest independent subset, which we denote by $r(S)$.  We define
the {\em closure} of a set $S$ to be
\[
\cl{S}= \bigcup_{\substack{T\colon S\subseteq T\subseteq[n],\\
r(T)=r(S)}} T.
\]
A set $X\subseteq[n]$ for which $\cl{X}=X$ is
called a {\em flat} of $\M$.  Let $L(\M)$ denote the set of flats
$\M$, and $L_p(\M)$ the subset of flats of rank $p$, for $0\leq p\leq d$.
The set $L(\M)$ is partially ordered by inclusion.  In fact, $L(\M)$ is a
lattice, with operations
\begin{align*}
X\wedge Y =& X\cap Y\text{~and}\\
X\vee Y = & \cl{X\cup Y}.
\end{align*}
It turns out that the lattice $L(\M)$ is geometric, and all
geometric lattices arise in this way (for which, see \cite[\S 1.7]{Oxbook}).
$L(\M)$ has a minimal element, $\bottom:=\cl{\emptyset}$.

If $\M$ is a matroid on $E$ with bases $\B$, the {\em dual matroid}
$\M^*$ is the matroid with bases $\set{E-B\colon B\in\B}$.  In particular,
$\rank(\M)+\rank(\M^*)=n$.  Clearly the construction of duals is an
involution, and it abstracts the idea of orthogonal complements: see
Example~\ref{ex:dual} below.

Finally, an element $x\in[n]$ for which $r(\set{x})=0$ is a {\em loop}, and
distinct elements $x,y\in[n]$ are {\em parallel} if $r(\set{x,y})=1$.  A
matroid $\M$ is {\em simple} if it has no loops or parallel edges.
\subsection{Restrictions and sums}\label{ss:restrs}
If $X\subseteq[n]$, then the set $\set{I\cap X\colon I\in\In(\M)}$
defines the independent sets of a matroid on $X$.  This is denoted $\M|X$, the
{\em restriction} of $\M$ to $X$.  If $X$ is a flat of $\M$, then
flats of $\M|X$ are exactly those
flats of $\M$ contained in $X$, which is to say that $L(\M|X)=[\bottom,X]$,
an interval in $L(\M)$.  

Dually, for $X\subseteq[n]$, the {\em contraction} of $\M$ by $X$, written
$\M/X$, is the matroid on the set $[n]-X$ defined by the rank
function
\[
r_{\M/X}(S):=r_{\M}(S\cup X)-r_{\M}(X),
\]
for $S\subseteq [n]-X$.
The flats of $\M/X$ are in bijection with flats of $\M$ containing $X$, so
$L(\M/X)\cong [\cl{X},[n]]$: see \cite[Prop.~3.3.8]{Oxbook}.  If $X$ is not a 
flat of $\M$, then any element $e\in \cl{X}-X$ is a loop in $\M/X$.

If $\M_1$ and $\M_2$ are matroids on $E_1$ and $E_2$, respectively,
the sum $\M=\M_1\oplus \M_2$ is defined to have bases
\[
\B=\set{B_1\cup B_2\colon B_1\in \B(\M_1),\, B_2\in \B(\M_2)}.
\]
Alternatively, the independence complex of $\M$ is the simplicial
join $\In(M_1)*\In(M_2)$, and $L(\M)=L(\M_1)\times L(\M_2)$.
  A matroid is {\em connected} if it cannot
be (nontrivially) decomposed as a sum of matroids.  Both $E_1$ and $E_2$
are flats of $\M_1\oplus \M_2$; conversely, $\M$ is connected provided that
there are no sets $X\subseteq[n]$ for which $X$ and $[n]-X$ are both flats
of $\M$.  The obvious
notion of connected components is well-defined, and we
denote the number of them by $\kappa(\M)$.  A flat $X$ is
said to be {\em irreducible} if the restriction $\M|X$ is connected.
We let $L_\irr(\M)$ denote the set of irreducible flats.

\subsection{Polytopes}\label{ss:polytopes}
There is a remarkable relationship between matroids,
convex geometry and toric varieties, discovered
by Gel$'$fand, Goresky, MacPherson, and Serganova \cite{GGMS87}.
In order to begin to describe it,
let $e_1,\ldots,e_n$ be the standard basis for $\Z^n$, and for any
subset $S\subseteq[n]$, let $e_S=\sum_{i\in S}e_i$.

If $\M$ is a matroid on $E=[n]$, let
\[
P_{\M}=\conv\set{e_B\colon B\in\B}\subseteq\R^n,
\]
the convex hull of the indicator vectors on the bases of $\M$.
This is the {\em matroid
polytope} of $\M$.
If $\M$ has rank $d$, then $P_{\M}\subseteq d\cdot\Delta^{n-1}$, where
$\Delta^{n-1}=\conv(\set{e_1,\ldots,e_n})\subseteq\R^n$, the
standard $(n-1)$-dimensional simplex.  

If $\M$ is connected, then
$P_{\M}$ has dimension $n-1$.  If $\M$ decomposes as a sum, then 
$P_{\M}$ is a corresponding cartesian product, so the dimension 
of $P_{\M}$ is $n-\kappa(\M)$ (see~\cite{FS05}).  
Faces of the matroid polytope are
themselves matroid polytopes.   Such matroids are obtained from $\M$
by deleting independent sets, so are called {\em degenerations} of $\M$.
If $F$ is a face of $P_{\M}$, let $\M_F$ denote the matroid whose
bases are the vertices of the face $F$.  For a linear functional $u\in\R^n$,
let $F$ be the face of $P_{\M}$ on which $u$ achieves its minimum,
and let $\M_u$ denote the matroid corresponding to the face: then
$F=P_{\M_u}$.
%(If $\M$ is connected, then the codimension-$k$ faces of $P_{\M}$
%correspond to matroids with $k$ components.)

The face structure of the matroid polytope was worked out as follows
in \cite{AK06,FS05}.  We assume that $\M$ is connected.  For $u\in \R^n$, let
\begin{equation}\label{eq:Fu}
\F(u):=(\emptyset=F_0\subset F_1\subset \cdots \subset F_k=[n])
\end{equation}
be the (unique) chain of subsets for which the coefficients
$\set{u_i\colon i\in F_a-F_{a-1}}$ are constant for all $1\leq a\leq k$,
and $u_i\leq u_j$ whenever $i\in F_a$ and $j\in F_b$ for $a<b$.

Then the face of $P_{\M}$ that minimizes $u$ depends only on $\F(u)$,
and
\[
\M_u=\bigoplus_{i=1}^k (M|F_i)/F_{i-1}.
\]

\begin{theorem}[\cite{AK06,FS05}]\label{thm:flacet0}
The following are equivalent for a vector $u\in\R^n$:
\begin{itemize}
\item The matroid $\M_u$ has no loops;
\item The subsets in the chain $\F(u)$ are flats of $\M$;
\item The face $F=P_{\M_u}$ is not contained in the boundary of the simplex
$\partial(d\cdot\Delta^{n-1})$.
\end{itemize}
If $i$ is a loop in $\M_u$, then $e_i$ is an (inner) normal vector for
$F$.
\end{theorem}
%For $u\in\R^n$, the matroid $\M_u$ has no loops if and only if 
%the subsets in the chain $\F(u)$ are flats of $\M$.
%\end{theorem}
Clearly the number of components 
$\kappa(\M_u)\geq k$; if in fact this is an equality,
then $P_{\M_u}$ is face of codimension $k-1$ in $P_{\M}$.  In the case
where $k=2$, we obtain a description of facets.  That is, 
if $\M_u$ contains no loop, then $\F(u)=(\emptyset\subset X
\subset[n])$ for some proper flat $X$.  We may shift $u$ by a multiple
of the vector $e_{[n]}$ and rescale by a positive number without changing
the face $P_{\M_u}$, so we may assume that $u=-e_X$.
\begin{theorem}[\cite{FS05}]\label{thm:flacet}
%A face $F$ of $P_{\M}$ is contained in the boundary
%$\partial(d\cdot\Delta^{n-1})$ if and only if $\M_F$ contains a loop $i$.
%In that case, the vector $e_i$ is an (inner) normal vector for $F$.
A facet $F$ of $P_{\M}$ which is not contained in $\partial(d\cdot\Delta^{n-1})$
has an inner normal vector $-e_X$ for some flat $X\in L(\A)$.  Such
facets are in bijection with the set
\begin{equation}\label{eq:flacet}
\set{X\in L(\A)\colon \M|X\text{~and~}\M/X\text{~are both connected.}}
\end{equation}
\end{theorem}
For reasons which will be more evident later, it will be convenient to 
think of outer normal vectors of matroid polytopes, instead of inner ones.
\begin{definition}
For a matroid $\M$, 
let $\Sigma_{\M}$ denote the outer normal fan of the matroid polytope
$P_{\M}$.  By definition, this is an equivalence relation on vectors
in $\R^n/\R(1,1,\ldots,1)$, where $u\sim v$ if and only if $u$ and $v$
achieve their maximum value on the same face of $P_{\M}$.  The
equivalence classes form the relative interiors of polyhedral cones:
we refer to \cite[\S 7.1]{Zbook} for background.
\end{definition}

\begin{example}\label{ex:uniform0}
For integers $1\leq d\leq n$, the {\em uniform matroid} $U_{d,n}$ has
underlying set $[n]$, and its bases consist of all subsets of $[n]$ of
size $d$.  The matroid polytope of $U_{2,4}$ is a $3$-dimensional polytope
in $\R^4$.  It is shown in Figure~\ref{fig:octahedron} with the facets
of the form \eqref{eq:flacet} shaded.
\end{example}
\begin{figure}
%\centering
\subfigure[$P_{\M}$ and some degenerations]{
\begin{tikzpicture}[scale=2,line join=bevel,z={(-0.2,-0.12)},text opacity=1.0]
\coordinate[label=above right:$13$] (A13) at (0,0,-1) {};
\coordinate[label=left:$23$] (A23) at (-1,0,0);
\coordinate (A24) at (0,0,1);
\coordinate[label=right:$14$] (A14) at (1,0,0);
\coordinate[label=above:$12$] (A12) at (0,1,0);
\coordinate[label=below:$34$] (A34) at (0,-1,0);

\coordinate (O) at (0,0,0);
\coordinate (F1) at (1/3,1/3,-1/3);
\coordinate (F2) at (-1/3,1/3,1/3);
\coordinate (F3) at (-1/3,-1/3,-1/3);
\coordinate (F4) at (1/3,-1/3,1/3);
\coordinate (E1) at (-1/3,-1/3,1/3);
\coordinate (E2) at (1/3,-1/3,-1/3);
\coordinate (E3) at (1/3,1/3,1/3);
\coordinate (E4) at (-1/3,1/3,-1/3);

\coordinate (S1) at (1,1,-1);
\coordinate (S2) at (-1,1,1);
\coordinate (S3) at (-1,-1,-1);
\coordinate (S4) at (1,-1,1);

%draw outer simplex
%\foreach \c in {(A12), (A13), (A14)} {
%  \draw[style=dashed,thin,color=gray] (S1) -- \c;
%}
%\foreach \c in {(A13), (A23), (A34)} {
%  \draw[style=dashed,thin,color=gray] (S3) -- \c;
%}

%\coordinate (F4mat) at ($ (O)!2.5!(F4) $)

%\node (F4mat) at ($ (O)!2.5!(F4) $) {$\tiny\begin{pmatrix} 1 & -1 & 2 & 0\\ 0 & 0 & 0 & 1\end{pmatrix}$};

% draw little matroids
\node [rectangle,draw,inner sep=2pt,fill=white] (E3mat) at ($ (O)!3!(E3) $) {
\begin{tikzpicture}
\node (m1) [circle,draw,inner sep=1pt] at (0,0) [label=above:{\tiny $3$}] {};
\node (m2) [circle,fill=black,inner sep=1pt] at (0.3,0) [label=above:{\tiny $1$}] {};
\node (m3) [circle,fill=black,inner sep=1pt] at (0.6,0) [label=above:{\tiny $2$}] {};
\node (m4) [circle,fill=black,inner sep=1pt] at (0.9,0) [label=above:{\tiny $4$}] {};
\draw (m2) -- (m4);
\end{tikzpicture}
};
\node [rectangle,draw,inner sep=2pt,fill=white] (F4mat) at ($ (O)!3!(F4) $) {
\begin{tikzpicture}
\node [circle,fill=black,inner sep=1pt] at (0,0) [label=above:{\tiny $1,2,3$}] {};
\node [circle,fill=black,inner sep=1pt] at (0.7,0) [label=above:{\tiny $\phantom{1,}4\phantom{,2}$}] {};
\end{tikzpicture}
};

\node [rectangle,draw,inner sep=2pt,fill=white] (E1mat) at ($ (O)!3!(E1) $) {
\begin{tikzpicture}
\node (m1) [circle,draw,inner sep=1pt] at (0,0) [label=above:{\tiny $1$}] {};
\node (m2) [circle,fill=black,inner sep=1pt] at (0.3,0) [label=above:{\tiny $2$}] {};
\node (m3) [circle,fill=black,inner sep=1pt] at (0.6,0) [label=above:{\tiny $3$}] {};
\node (m4) [circle,fill=black,inner sep=1pt] at (0.9,0) [label=above:{\tiny $4$}] {};
\draw (m2) -- (m4);
\end{tikzpicture}
};

\node [rectangle,draw,inner sep=2pt,fill=white] (F2mat) at ($ (O)!3!(F2) $) {
\begin{tikzpicture}
\node [circle,fill=black,inner sep=1pt] at (0,0) [label=above:{\tiny $1,3,4$}] {};
\node [circle,fill=black,inner sep=1pt] at (0.7,0) [label=above:{\tiny $\phantom{1,}2\phantom{,2}$}] {};
\end{tikzpicture}
};

% draw octahedron
\draw (A13) -- (A23) -- (A12) -- cycle;
\draw [fill opacity=0.5,fill=orange!80!black](A14) -- (A13) -- (A12) -- cycle;
\draw [fill opacity=0.5,fill=green!80!blue](A13) -- (A23) -- (A34) -- cycle;
\draw (A14) -- (A13) -- (A34) -- cycle;
\draw [fill opacity=0.5,fill=green!30!black](A23) -- (A24) -- (A12) -- cycle;
\draw  (A24) -- (A14) -- (A12) -- cycle;
\draw  (A23) -- (A24) -- (A34) -- cycle;
\draw  [fill opacity=0.5,fill=purple!70!black](A24) -- (A14) -- (A34) -- cycle;

% front bits of the tetrahedron
%\foreach \c in {(A12), (A23), (A24)} {
%  \draw[style=dashed,thin,color=gray] (S2) -- \c;
%}
%\foreach \c in {(A14), (A24), (A34)} {
%  \draw[style=dashed,thin,color=gray] (S4) -- \c;
%}

%\draw[->,style=thin,color=gray] (F1) -- ($ (0)!2.5(F1) $);
\draw[->,style=thin] (F2) -- ($ (O)!2.5!(F2) $);
%\draw[->,style=thin,color=gray,opacity=1.0] (F3) -- ($ (O)!2.5!(F3) $);
\draw[->,style=thin] (F4) -- (F4mat);

%\node at (F4mat) {$\tiny\begin{pmatrix} 1 & -1 & 2 & 0\\ 0 & 0 & 0 & 1\end{pmatrix}$};

\draw[->,style=thin] (E1) -- ($ (O)!2.5!(E1) $);
%\draw[->,style=thin,color=gray] (E2) -- ($ (O)!2.5!(E2) $);
\draw[->,style=thin] (E3) -- ($ (O)!2.5!(E3) $);
%\draw[->,style=thin,color=gray] (E4) -- ($ (O)!2.5!(E4) $);
%\draw[->,style=thin,color=gray] (E4) -- (E4mat);

\node at (A24) [label=below left:$24$] {};

\end{tikzpicture}
}
\qquad
\subfigure[Rays of the fan $\Sigma_{\M}$]{
\raise25pt\hbox{$
\begin{tikzpicture}[scale=2,line join=bevel,z={(-0.2,-0.12)},text opacity=1.0,
plain/.style={->,thin}]
\coordinate (z) at (2/3,2/3,2/3);
\coordinate (1) at (2/3,2/3,-2/3);
\coordinate (2) at (2/3,-2/3,2/3);
\coordinate (12) at (2/3,-2/3,-2/3);
\coordinate (3) at (-2/3,2/3,2/3);
\coordinate (13) at (-2/3,2/3,-2/3);
\coordinate (23) at (-2/3,-2/3,2/3);
\coordinate (123) at (-2/3,-2/3,-2/3);

\foreach \c in {(z),(12),(13),(23)} {
  \draw[style={->,dotted,thick}] (0,0,0) -- \c;
}

\foreach \c in {(1),(2),(3),(123)} {
  \draw[style={->,very thick}] (0,0,0) -- \c;
}

\draw[style=very thin,dashed,color=gray]
(z) -- (1) -- (12) -- (2) -- (23) -- (3) -- (13) -- (123) -- (23);
\draw[style=very thin,dashed,color=gray]
(1) -- (13);
\draw[style=very thin,dashed,color=gray]
(2) -- (z) -- (3);
\draw[style=very thin,dashed,color=gray]
(123) -- (12);
\end{tikzpicture}
$}\label{fig:octahedron_b}
}

\caption{The matroid polytope of $U_{2,4}$}\label{fig:octahedron}
\end{figure}
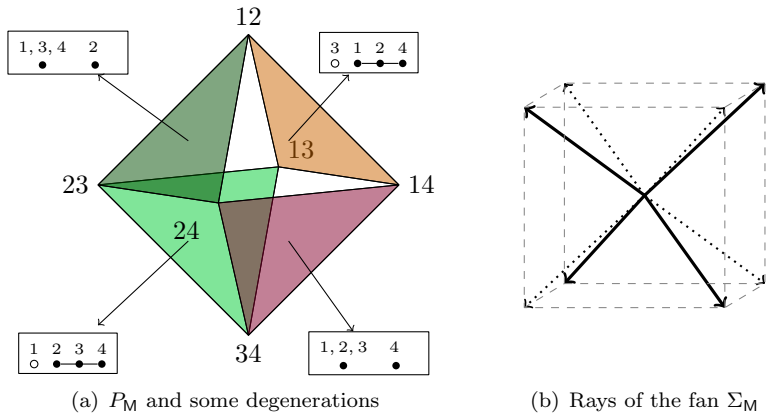
\begin{remark}
Various other polytopes can be associated with a matroid, and the reader
should see \cite[Ch.~40]{Sch03} for details.  In particular, the independent
set polytope is the convex hull of indicator functions of all independent
sets, and $P_{\M}$ is its facet with normal vector $e_{[n]}$.
\end{remark}

\begin{exercise}\label{ex:dual_polytope}
Show that if $\M$ is a matroid on the set $[n]$, then the matroid polytopes
for $\M$ and its dual are related by $P_{\M^*}=e_{[n]}-P_{\M}$, where
``$-$'' denotes reflection through the origin.  Conclude the outer normal fan of
$P_{\M}$ is the inner normal fan of $P_{\M^*}$.
\end{exercise}

\subsection{Linear matroids}\label{ss:linear}
% do linear realizations first
A hyperplane arrangement is, informally, a linear realization of a matroid
with no loops.  (Our standard reference for hyperplane arrangements is the
book \cite{OTbook}.)  In order to make this informal notion precise,
let $\k$ be an algebraically closed field, and let $V$ be a $d$-dimensional
linear subspace of $\k^n$, where $d\geq1$
If we let $f\colon V\hookrightarrow\k^n$
denote the inclusion, its $i$th coordinate is a linear map, $f_i\in V^*$.
$V$ determines a matroid $\M(V)$ on the set of coordinates $[n]$:
one declares a set $I\subseteq [n]$ to be an independent set of 
$\M(V)$ if and only
if the set $\set{f_i\colon i\in I}$ is linearly independent.  Then the
remaining matroid-theoretic terms have straightforward linear-algebraic
counterparts.  Such a matroid is said to be {\em linear}, or a {\em
linear realization} of its isomorphism type.

Let $x_1,\ldots,x_n$ be coordinate functions on $\k^n$, and let
$\CC_n=\big\{\hat{H}_1,\ldots,\hat{H}_n\big\}$ denote the set of
coordinate hyperplanes, $\hat{H}_i=\ker x_i$ for $1\leq i\leq n$.
Note that $V\subseteq\hat{H}_i$ if
and only if $f_i=0$, which is equivalent to $i$ being a loop of $\M(V)$.
Provided that $\M(V)$ has no loops, $H_i:=\hat{H_i}\cap V$ is a hyperplane of
$V$ for all $i$.  Let $\A=\A(V)=\set{H_1,\ldots,H_n}$.  This is
a (central, essential) hyperplane arrangement in $V$.  By construction,
$H_i=\ker f_i$ for $1\leq i\leq n$, so $\set{i,j}$ is dependent in
$\M(V)$ if and only if $H_i=H_j$.  Thus a simple
matroid $\M(V)$ gives an arrangement of distinct hyperplanes.
Conversely, a set of $n$ hyperplanes in a $d$-dimensional
linear space $V$ is called {\em essential} if $\bigcap_{i=1}^n H_i=\set{0}$,
and in this case, any
choice of defining equations for the hyperplanes give a linear
realization of a simple matroid with no loops.

\begin{example}\label{ex:uniform}
Choose integers $1\leq d\leq n$, and let $A=A_{d,n}$ be a $d\times n$ matrix
for which any $d$ columns are linearly independent.  Let $V=\row(A)$, the
span of the rows.  The matroid $\M(V)$ is a realization of the
{\em uniform matroid} $U_{d,n}$.  The independent sets consist of all
subsets of $[n]$ of cardinality at most $d$.  The hyperplane arrangement $\A$
is called
the {\em generic arrangement} of $n$ hyperplanes in $\k^n$.  If $n=d$, then
$V=\k^n$, and $\A=\CC_n$, called the {\em Boolean arrangement} of rank $n$.
If $d=2$, $\A$ consists of $n$ lines in the plane through the origin.
\end{example}

\begin{example}\label{ex:delA3}
Let $V=\row(A)$, where
\[
A=\begin{pmatrix}
1 & 0 & 0 & 1 & 1 \\
0 & 1 & 0 & -1 & 0 \\
0 & 0 & 1 & 0 & -1
\end{pmatrix};
\qquad\text{$\M(V)$:}\quad
\def\dx{0.5}\def\dy{0.707}
\raise10pt\hbox{$
\begin{tikzpicture}[scale=0.75,baseline=(current bounding box.center),
plain/.style={circle,draw,inner sep=1.5pt,fill=black}]
\node[plain] (1) at (0,0) [label=above:$1$] {};
\node[plain] (2) at (-\dx,-\dy) [label=above left:$2$]{};
\node[plain] (4) at (-2*\dx,-2*\dy) [label=above left:$4$] {};
\node[plain] (3) at (\dx,-\dy) [label=above right:$3$] {};  
\node[plain] (5) at (2*\dx,-2*\dy) [label=above right:$5$] {};
\draw (1) -- (2) -- (4);
\draw (1) -- (3) -- (5);
\end{tikzpicture}
$}
\qquad
\text{$\A(V)$:}\quad
\lower30pt\hbox{$
\begin{tikzpicture}[scale=0.9]
%\clip (-1,1.3) -- (1,1.3) -- (0,-0.7) -- (-1,1.3);
\begin{scope}
\clip (0,0.5) circle (1.3);
\coordinate (UL) at (-0.7,1);
\coordinate (UR) at (0.7,1);
\coordinate (B) at (0,-0.35);
\coordinate (ML) at ($ (UL) ! 0.5 ! (B)$ );
\coordinate (MR) at ($ (UR) ! 0.5 ! (B)$ );

\newcommand{\drawline}[2]
{
\draw ($ #1 ! -0.5 ! #2 $) -- ($ #1 ! 1.5 ! #2 $);
}

\drawline{(UL)}{(UR)}
\drawline{(UL)}{(B)}
\drawline{(UR)}{(B)}
\drawline{(UR)}{(ML)}
\drawline{(UL)}{(MR)}
\end{scope}
\node at ($ (UL) ! -0.5 ! (UR) $) {$1$};
\node at ($ (UL) ! 1.7 ! (MR) $) {$2$};
\node at ($ (UL) ! -0.5 ! (B) $) {$4$};
\node at ($ (UR) ! -0.5 ! (B) $) {$5$};
\node at ($ (UR) ! 1.7 ! (ML) $) {$3$};
\end{tikzpicture}
$}  % end hbox
\]
where the matroid diagram encodes the flats as collinear subsets, and
the arrangement $\A(V)$ is drawn in a suitable affine chart of $\P^2$.
Here, $n=5$, $d=3$, and the linear functionals are $\set{u,v,w,u-v,u-w}$.
% matroid
\end{example}
\begin{example}\label{ex:braid}
For $d\geq2$, let $V=V(A_d)=\C^{d+1}/\C e_{[d+1]}$, where
$e_{[d+1]}=(1,1,\ldots,1)$, and let
$f_{ij}=z_i-z_j$ for $1\leq i<j\leq d+1$.  Then $f\colon V\to \C^n$ is a
linear inclusion, where $n={d\choose 2}$, and $\M(V)$ is the matroid
of the complete graph, $K_{d+1}$.  Moreover,
\[
U(\A)=\set{z\in V\colon z_i\neq z_j\text{~for $i\neq j$}}
\]
is the complement of the reflecting hyperplanes of the $A_d$ root system,
also known as the (complex) {\em braid arrangement}.  One also writes $U(\A)$ as
$F(\C,d+1)$, the configuration space of $d+1$ (distinct, labelled) points
in $\C$.

The flats $X$ of $\A$ are indexed by partitions $\pi_X$ of
the set $[d+1]$, by putting indices $i$ and $j$ in the same block of $\pi_X$
if and only if $z_i=z_j$ on $X$.  Then $X\leq Y$ if and only if
$\pi_X$ refines $\pi_Y$,
and the rank of $X$ equals $d$ minus the number of blocks in $\pi_X$.
A flat is irreducible if and only if $\pi_X$ has exactly one block of size
greater than one.
\end{example}

\begin{example}\label{ex:dual}
If $\M=\M(V)$ is linear, then the dual matroid $\M^*=\M(V^\perp)$, where
$V^\perp$ denotes the orthogonal complement of $V$ in the dual vector
space $(\k^n)^*$.
\end{example}

\begin{remark}\label{rem:grass}
It can be useful to think of $V$ as a point in the Grassmannian $\Gr_{d,n}$,
which embeds in $\P(\bigwedge^d \k^n)$ via the Pl\"ucker embedding.
Homogeneous coordinates for the latter space are indexed by subsets
$I\subseteq[n]$ of size $d$: let $x_I=x_{i_1}\wedge\cdots\wedge x_{i_d}$,
where $I=\set{i_1,\ldots,i_d}$ and $i_1<\ldots<i_d$.  Translating our
remarks above into this notation, we see the $I$th coordinate of $V$ is nonzero
if and only if $I$ is independent in $\M(V)$.
\end{remark}

\subsection{Linear restrictions and sums}\label{ss:lin_restrs}
%\begin{remark} % could have a parallel subsection, restr and sums
A linear matroid decomposes as a matroid
sum if and only if the vector space $V$ has a splitting which is
compatible with the coordinates in $\k^n$.  More precisely,
for any subset $S\subseteq[n]$, let
\[
\k^S=\set{x\in\k^n\colon x_i=0\text{~for $i\not\in S$}}
\]
denote the coordinate subspace.
If the matroid $\M(V)=\M_1\oplus \M_2$ where $\M_i$ is a matroid
on $E_i$, for $i=1,2$, it is not hard to check that $\M_i=\M(V_i)$,
where $V\cong V_1\oplus V_2$, and $V_i=V\cap \k^{E_i}$.  Conversely,
suppose $V_i\subseteq \k^{E_i}$ for $i=1,2$, where $[n]=E_1\dot\cup E_2$
is a partition with nonempty parts.  Let $V=V_1\oplus V_2$: then
$\M(V)=\M(V_1)\oplus \M(V_2)$.
%\end{remark}

Along the same lines, if $X\subseteq[n]$ is a flat of $\M=\M(V)$,
let $V_X$ denote the image of $V$ under the
coordinate projection map $\k^n\to\k^X$.   Then $V_X$ is a linear
quotient of $V$, and a realization of $\M|X$ in $\k^X$.

\subsection{A torus action}\label{ss:GM}

It should be clear that multiplying each coordinate map $f_i$ by a
nonzero scalar does not change the matroid or set of hyperplanes $\A$.
That is, let $T^n=(\k^\times)^n$
denote the algebraic torus of rank $n$, which acts on $\k^n$ by
(left) multiplication: then $\M(t\cdot V)=\M(V)$ for all $t\in T^n$.

Of course, if the coordinates of $t\in T^n$ are all equal,
$t$ acts on $\k^n$ by scalar multiplication, and $t\cdot V=V$.  So we
should instead consider the quotient torus $\PT^{n-1}:=T^n/\k^\times$
by the diagonal one-parameter subgroup.  In other words,
$T^n$ acts on the Grassmannian $\Gr_{d,n}$ via its
action on $\k^n$, and the action factors through $\PT^{n-1}$.

In this language, the matroid structure is constant on each torus orbit
$\PT^{n-1} \cdot V\subseteq \Gr_{d,n}$.
The hyperplane arrangements $\A(V)$ at different
points in the orbit have different ambient spaces, but they are linearly
isomorphic.
It remains to consider stabilizers of the torus action on $\Gr_{d,n}$.
It turns out that $\PT^{n-1}$ acts freely on $V$ when $\M(V)$ is connected:

\begin{proposition}\label{prop:Gr_orbit}
If $\M(V)$ has $\kappa$ connected components, then
\[
\stab_{T^n}(V)\cong(\k^\times)^\kappa.
\]
\end{proposition}
\begin{proof}
Suppose that $[n]=E_1\dot\cup\cdots\dot\cup E_\kappa$ is the decomposition into
connected components.  We claim
\begin{equation}\label{eq:stabV}
\stab_{T^n}(V)=\set{t\in T^n\colon t_i=t_j\hbox{~provided $i,j\in E_k$ for
some $k$}}.
\end{equation}
Since $V=V_1\oplus\cdots\oplus V_\kappa$ where $V_i=V\cap \k^{E_i}$ for
$1\leq i\leq \kappa$, scalar multiplication on each factor shows that
the right-hand side of \eqref{eq:stabV} is included in the left.
To show the other inclusion, suppose $t\in\stab_{T^n}(V)$.  The
eigenspaces of $t$ as an endomorphism of $\k^n$ are simply coordinate
subspaces, indexed by the partition of $[n]$ into subsets $F_1,\ldots,
F_l$ on which the coordinates of $t$ are constant (and pairwise
distinct).  By hypothesis, the action of $t$ restricts to $V$, so
\[
V=(V\cap \k^{F_1})\oplus\cdots\oplus (V\cap \k^{F_l}).
\]
It follows that the partition $E_1,\ldots,E_\kappa$ refines this one,
and $t$ is contained in the right-hand side of \eqref{eq:stabV}.
\end{proof}
In particular, if $\M(V)$ is connected, then the orbit $T^n\cdot V
\cong \PT^{n-1}$.  This provides an example of a torus torsor
in the Grassmannian.

\subsection{Arrangement complements}\label{ss:U}
For any hyperplane arrangement $\A$ in $V$, let
$U(\A)=V-\bigcup_i H_i$.  This is both the complement of the hypersurface
$f^{-1}(0)$ as well as an irreducible, closed subvariety of the torus
$T^n$, since $U(\A)=V\cap T^n$.  It is useful to keep in mind both
points of view.  The torus is, itself,
the complement of the Boolean arrangement $\CC^n$.  The space $U(\A)$ is a
central object of study in the theory of hyperplane arrangements,
particularly when $\k=\C$: we refer to the forthcoming book \cite{CDFSSTY}.

Let $\P U(\A):=U(\A)/\k^\times$, where $\k^\times$ is the diagonal subgroup
again.  This is a subvariety of the projective space $\P V$:
as above, $\P U(\A)=\P V\cap \PT^{n-1}$.
The quotient torus splits as a coordinate subtorus of $T^n$.
For $1\leq i\leq n$, let $T^{n-1}_i=\set{t\in T\colon t_i=1}$, and consider the
isomorphism of groups
$c_i\colon \PT^{n-1}\to T^{n-1}_i$, given by $c_i(t)=t_i^{-1}\cdot t$.
By restricting $c_i$ to $\P U(\A)$,
we see that $\P U(\A)$ is isomorphic to the subvariety
of $U(\A)$ on which $f_i=1$.
\begin{exercise}
Check that
$U(\A)\cong \P U(\A)\times \k^\times$.
\end{exercise}

\begin{remark}\label{rem:veryaffine}
An irreducible, closed subvariety of an algebraic torus is said to be
{\em very affine}.  Another very affine variety associated with a
hyperplane arrangement is given as follows.  Given a lattice vector
$m\in \Z^n$ with $\gcd\set{m_i\colon i\in [n]}=1$, form the subgroup
\[
T^m = \set{t\in T\colon t_1^{m_1}t_2^{m_2}\cdots t_n^{m_n}=1},
\]
and let $F(\A,m)=V\cap T^m$.  In general, this is a level set of a
{\em master function}
\[
\prod_{i=1}^n f_i^{m_i}\colon U(\A)\to \k^\times.
\]
See, for example, \cite{Var11}, for more details.
If $m_i>0$ for each $1\leq i\leq n$, the variety $F(\A,\m)$ is an (unreduced)
Milnor fibre of $\A$ with given multiplicities.  If each $m_i=1$, it
coincides with the usual, global Milnor fibre of the hypersurface.
These varieties are qualitatively quite different
from the hyperplane complements, and their topology is more subtle: see the
paper by Suciu in this volume~\cite{Su13}.  It would be interesting to
know if there is anything special about
the topology of intersections of linear spaces $V\subseteq\k^n$
and general subtori of $T^n$.
\end{remark}

\section{Toric varieties}\label{sec:toric}
For background on this subject,
we recommend the excellent book by Cox, Little and Schenck
\cite{CLS11}.  In order to keep these notes self-contained,
we will give a brief outline of the role of polyhedral combinatorics
in the theory of toric varieties.
From now on, we will assume that $\k=\C$.

We recall that a (normal) toric
variety of dimension $n$ to be an irreducible, normal complex variety $X$ for
which
\begin{itemize}
\item $X$ contains a dense, open subvariety isomorphic to a complex torus
$T^n$;
\item the action of $T^n$ on itself extends to $X$, giving
an algebraic map $T^n\times X\to X$.
\end{itemize}
\subsection{Cones and orbits}
One foundational part of the theory is that toric varieties are stratified
by closures of torus orbits, and that this stratification determines the
toric variety's isomorphism type.  The orbits are described combinatorially
by means of a polyhedral fan, denoted $\Sigma_X$, and we recall
briefly the nature of this description.

The one-parameter subgroups of $T^n$ form an integer lattice,
$N:=\Hom(\C^\times,T^n)$, which is isomorphic to $\Z^n$.  Suppose a
toric variety $X$ containing $T^n$ is given.  Then, for each
$u\in N$, we check whether or not
the homomorphism $u\colon\C^\times\to T^n\subseteq
X$ can be extended continuously to a map $u\colon \C\to X$.  That is,
let
\[
\abs{\Sigma_{X,\Z}}=\big\{u\in N\colon\text{
$\lim_{t\to0} u(t)$ exists.}\big\}
\]
Impose an equivalence relation by letting,
for $u,v\in\abs{\Sigma_{X,\Z}}$,
\[
u\sim v\text{~if and only if~}\lim_{t\to0} u(t)=\lim_{t\to0} v(t).
\]
If $S\subseteq N$ is an equivalence class, let $N_\R=N\otimes_\Z \R$,
a real Euclidean space, and define $\sigma\subseteq N_\R$ by
\[
\sigma=\overline{\R_{\geq0}\, S},
\]
where $\overline{\vphantom{i}\;\cdot\;}$ denotes closure in the usual topology.
It turns out that each such $\sigma$ is a polyhedral cone.

Let $\Sigma_X$
denote the set of all cones, which is (by construction) in bijection with
the set of limit points of one-parameter subgroups.  The set of cones
$\Sigma_X$ is closed under intersection.  For each $\sigma\in \Sigma_X$,
let $O(\sigma)=T^n\cdot\lim_{t\to 0}u(t)$ denote the torus orbit coming
from any $u\in\abs{\Sigma_{X,\Z}}\cap\sigma$.  All orbits arise this way, and
the cones keep track of their incidence data: that is,
\[
O(\sigma)\subseteq\overline{O(\tau)}\text{~if and only if~}\tau\subseteq\sigma.
\]
\begin{example}\label{ex:affine}
Affine space $\C^n$ is a toric variety whose fan $\Sigma_{\C^n}$ consists of
cones $\sigma_S$, for all $S\subseteq[n]$, where
$\sigma_S=\R_{\geq0}\set{e_i\colon i\in S}$, and $\set{e_1,\ldots,e_n}$ are
the standard basis for $N=\Z^n$.
The cone $\sigma_\emptyset=\set{0}\subseteq N$
corresponds to the maximal torus orbit $T^n\subseteq\C^n$ in the dictionary
above, while the full orthant $\sigma_{[n]}$ corresponds to the $0$-dimensional
orbit $\set{0}\subseteq\C^n$.
\end{example}
\begin{example}\label{ex:Pn}
The usual embedding $\PT^{n-1}\subseteq\P^{n-1}$ makes projective space a toric
variety with $N=\Z^n/\Z e_{[n]}$.  The toric fan $\Sigma_{\P^{n-1}}$
consists of cones $\sigma_S\subseteq \R^n/\R e_{[n]}$, for
all $S\subsetneq [n]$, where $\sigma_S$ is defined as in
Example~\ref{ex:affine}.  In particular, the one-dimensional cones
$\sigma_{\set{i}}$ correspond to orbits of points $[x_1\colon\cdots\colon x_n]
\in\P^{n-1}$ with $x_i=0$, for $1\leq i\leq n$.
\end{example}

A cone $\sigma$ of dimension $d$ is {\em simplicial} if it is spanned by $d$
rays.  A necessary condition for a toric variety $X$
to be smooth is that each cone of the fan $\Sigma_X$ is simplicial.
A sufficient condition for $X$ to be smooth is that $\Sigma_X$ is simplicial,
and that each cone $\sigma$ is {\em unimodular}: that is, $\sigma$
is generated by lattice vectors that extend to a
basis of the lattice $N$.

\subsection{Projective toric varieties}
An important family of toric varieties is constructed in the following
way. Let $A$ be a $n\times r$ integer matrix, and
consider the group homomorphism between tori, $p_A\colon T^n\to T^r$,
given by
\[
p(t)=(\prod_{i=1}^n t_i^{A_{i1}},\prod_{i=1}^n t_i^{A_{i2}},\ldots,
\prod_{i=1}^n t_i^{A_{ir}}).
\]
Then $p_A$ is injective provided that $A$ is unimodular of full rank, and $p_A$
induces a map of quotient tori $\bar p_A\colon \PT^{n-1}\to\PT^{r-1}$
provided that the sum of the entries in any two columns of $A$ are
equal.  If both of these conditions are satisfied, we have an embedding
\[
\bar p_A\colon \PT^{n-1}\to \PT^{r-1}\subseteq\P^{r-1}:
\]
let $X_A$ denote the toric variety given by taking the closure of the image
of $\bar p_A$.

If one translates each column of $A$ by a fixed lattice vector to obtain
a matrix $A'$, then $\bar p_A=\bar p_{A'}$, so $X_A=X_{A'}$.
Let $P_A$ denote the convex hull of the columns of $A$, a polytope in $\Z^n$.
Choose a translation of the columns of $A$ so that the column
sums are zero, so that $P_A\subseteq M:=N^*\cong \Z^{n-1}$.

In fact, any toric variety $X$ which is equivariantly
embedded in projective space can be written as $X=X_A$ for some $A$,
so we will simply write $P_X$ for (the
translation equivalence class of) the polytope $P_A$.

The toric fan of $X_A$ is easy to describe: provided that $X_A$ is normal,
the fan is the (inner) normal
fan of $P_A$.  This is a complete fan: $\abs{\Sigma_{X,\Z}}=N$, and for
$u,v\in N$,
we have $u\sim v$ if and only if $u$ and $v$, regarded as linear functionals
on $M$, both achieve their minimum on the same subset of $P_A$.

With this in mind, recall the Gel$'$fand-MacPherson construction
from \S\ref{ss:GM} that embedded a torus in $\Gr_{d,n}$.  If a linear
matroid $\M(V)$ is connected, then the map $a_V\colon \PT^{n-1}\to
\Gr_{d,n}$ defined by $a_V(t)=t\cdot V$ is injective, by
Proposition~\ref{prop:Gr_orbit}.  So the closure
of the image is a toric variety,
\[
X_{\M}:=\overline{a_V(\PT^{n-1})}\subseteq \Gr_{d,n}.
\]
The Grassmannian is complete, so $X_{\M}$ is too.  In fact, using the
Pl\"ucker embedding, we can identify
$X_{\M}$ with the closure of the image $\PT^{n-1}$ in the projective space
$\P(\bigwedge^d \C^n)$, as in Remark~\ref{rem:grass}, so $X_{\M}$ is
a projective toric variety.  It has a particularly nice description:

\begin{theorem}[\cite{GGMS87}]\label{thm:GGMS}
The weight polytope of $X_{\M}$ is the matroid polytope, $P_{\M}$.
\end{theorem}
\begin{proof}
We assume for simplicity that $\M(V)$ is connected, the general case
being similar.  Consider the composition of maps $b_V$,
\[
\xymatrix{
\PT^{n-1}\ar[r]^{a_V}\ar`d[r]`r[rrr]_{b_V} [rrr]
 & \Gr_{d,n}\ar[r] & \P(\bigwedge^d\C^n)\ar@{-->}[r] &
\P(\C^{\B}),
}
\]
%\ar@/_1.8pc/[rrr]_{b_V}
where last rational map is projection onto those $x_I$ for which $I$
is a base of $\M(V)$.  Its restriction to the image of the torus is, in
fact, an injective, equivariant regular map, and we want to calculate
the weights of the embedding of $X_{\M}$ in $\P(\C^{\B})$.

Suppose $V=\row(A)$ for some $d\times n$ matrix $A$.  For $t\in\PT^{n-1}$,
a basis for
$t\cdot V$ is obtained by multiplying column $j$ of $A$ by $t_j$.  For
a set $I\subseteq [n]$ of size $d$, we write $\det(A_I)$ for the minor
on columns $I$, which is also the $I$th Pl\"ucker coordinate of $V$.
Then for $I\in\B$, the columns of $A$ are independent, so $\det(A_I)\neq0$,
and the torus action on the $I$th coordinate is given by
\[
b_V(t\cdot V)_I=\big(\prod_{j\in I}t_j\big)\cdot b_V(V)_I.
\]
Reading the exponents of this monomial,
we see that the weight vector for coordinate $I$ is the vector
$e_I$, and these are the vertices of the matroid polytope.
\end{proof}

By a result of White~\cite{Wh77}, the toric variety $X_{\M}$ is normal.
So the closures of torus orbits in $X_{\M}$ are in bijection with the
faces of the matroid polytope $P_{\M}$.  Since a face $F$ of $P_{\M}$
is the matroid polytope of a matroid $\M_F$, the closure of the
corresponding orbit is another toric variety, $X_{\M_F}$.
We illustrate with an example.
\begin{example}\label{ex:uniform_polytope}
Continuing Example~\ref{ex:uniform}, consider the arrangement of four
lines in the plane, $\M(V)\cong U_{2,4}$.
If we choose $f\colon\C^2\to\C^4$, to be $f=(z,w,z-w,z+w)$,
then for $t\in\PT^3$, we have
\[
t\cdot V = \row
\begin{pmatrix}
t_1 & 0 & t_3 & t_4\\
0 & t_2 & -t_3 & t_4
\end{pmatrix},
\]
whose image in $\P^5$ is
\[
[t_1t_2\colon -t_1t_3\colon -t_2t_3\colon t_1t_4\colon -t_2t_4\colon 2t_3t_4].
\]
Then the matrix of weights, with the columns ordered in this way, is
\[
A=\begin{pmatrix}
1 & 1 & 0 & 1 & 0 & 0\\
1 & 0 & 1 & 0 & 1 & 0\\
0 & 1 & 1 & 0 & 0 & 1\\
0 & 0 & 0 & 1 & 1 & 1
\end{pmatrix},
\]
and the convex hull of the columns is the octahedral matroid polytope
of Figure~\ref{fig:octahedron}.
The normal fan of an octahedron is the fan in $N=\Z^4/\Z e_{[4]}$
obtained by taking cones over each face of the polar polytope, the cube.
The fan is not simplicial, so the toric variety $X_{\M}$ is not smooth.

The coordinate $x_1$ takes its minimum on the facet $F=\conv\set{e_{23},e_{24},
e_{34}}$, so the vector $u=(1,0,0,0)\in N$ spans a ray of the fan.
We can work out the corresponding codimension-$1$ torus orbit
by computing $\lim_{t\to 0} u(t)$, which is
\begin{equation}\label{eq:limit}
\lim_{t\to 0}\row
\begin{pmatrix}
t^1 & 0 & t^0 & t^0\\
0 & t^0 & -t^0 & t^0
\end{pmatrix}
=
\row\begin{pmatrix}
0 & 1 & -1 & 1\\
0 & 0 & 1 & 1\\
\end{pmatrix}.
\end{equation}
The matroid of this limit subspace has a loop, the element $1$.  By
Theorem~\ref{thm:flacet0}, the facet $F$ must lie in the boundary of the
simplex $2\cdot\Delta^3$, which it does.  (See Figure~\ref{fig:octahedron}).
\end{example}

\begin{example}\label{ex:delA3_again}
Recall the hyperplane arrangement from Example~\ref{ex:delA3} with $d=3$ and
$n=5$.  Its matroid polytope $P_{\M}$
is $4$-dimensional, and $\abs{\B}=8$.  Calculating by hand (or with help from
Macaulay~2~\cite[Polyhedra package]{M2}), one finds $P_{\M}$ has $18$ edges,
$17$ $3$-faces, and $7$ facets: see Figure~\ref{fig:loopy}.
One facet $F$ is contained in
$\partial(3\cdot\Delta^4)$
and has vertices $\set{e_{234},e_{235},e_{245},e_{345}}$.
Again, $1$ is a loop in $\M_F$, and $-e_1$ is an outer normal vector.
Four facets, normal to $e_i$ for $i=2,3,4,5$ are square-based pyramids.
The two facets normal to $e_{124}$ and $e_{135}$ are triangular prisms.
The vector $e_1$ takes its maximum
on their intersection, which is a square ($2$-dimensional) face, reflecting
the fact that the contraction $\M/\set{1}$ is not connected: see
Theorem~\ref{thm:flacet0}.
\end{example}

\begin{comment}
Since $\Gr_{2,4}$ is compact, the limit $\lim_{t\to 0}\col(A_u)$ exists
for all $u\in N$, and its value
depends only
on the order of the integers $u_1$, $u_2$, $u_3$ and $u_4$.  Of course
if $u=0$, then $\col(A_u)=V$.
On the other hand, if $u_1=u_2=u_3<u_4$ or $u_4<u_1=u_2=u_3$, we find
$\lim_{t\to 0}\col(A_u)$ is spanned by the rows of, respectively,
\[
\begin{pmatrix} 1 & -1 & 2 & 0\\ 0 & 0 & 0 & 1\end{pmatrix}\text{~and~}
\begin{pmatrix} 1 & 0 & 1 & 0\\ 0 & 1 & -1 & 0\end{pmatrix}.
\]
In both cases the orbit $\PT^3\cdot\col(A_u)$ is $2$-dimensional, since
the matroid of the first space is disconnected
(Proposition~\ref{prop:Gr_orbit}), and the matroid of the second contains a
loop (coordinate $4$).
Let $X$ denote the toric variety obtained as the closure of $i(T_0^3)\cdot V$
in $\Gr_{2,4}$: then continuing this analysis shows that the fan
$\Sigma_X$ in $N_\R$ can be described as the set of cones over the
faces of the cube in $\R^4/\R(1,1,1,1)$ with vertices
\begin{gather*}
\set{(1,0,0,0),(0,1,0,0),(0,0,1,0),(0,0,0,1),
(0,1,1,1),(1,0,1,1),(1,1,0,1),(1,1,1,0)}.
\end{gather*}
Figure~\ref{fig:octahedron} shows the rays of $\Sigma_X$ together with
some of the matroids of the limit points.
\end{example}
\end{comment}
We conclude the section with another example of how constructions in 
convex geometry are reflected by toric varieties.
Let $X_1$ and $X_2$ be $n$-dimensional projective toric varieties,
and $i_j\colon T^n\to X_j$ the inclusion of their maximal torus, for $j=1,2$.
Consider the diagonal map $d\colon T^n\to X_1\times X_2$ given by
$d(t)=(d_1(t),d_2(t))$,
and let $X=\overline{d(T^n)}\subseteq X_1\times X_2$.
\begin{proposition}[Prop.~8.1.4, \cite{GKZbook}]\label{prop:minkowski}
The space $X$ is also a projective toric variety, and
$P_X=P_{X_1}+P_{X_2}$,
where ``$+$'' denotes the Minkowski sum of polytopes.
\end{proposition}
\begin{exercise}%[June Huh]
Define a map $\PT^n\to \P^n\times\P^n$ by
$t\mapsto(t,t^{-1})$, and let $X_n\subseteq\P^n\times\P^n$ denote the closure
of its image.  Check that $X_n$ is a toric variety with
$P_{X_n}=\Delta^n+(-\Delta^n)$,
where $\Delta^n\subseteq \Z^{n+1}$ denotes the standard simplex,
and $-\Delta^n$ its reflection through the origin.  Show that $X_2$ is
isomorphic to a blowup of $\P^2$ at three points, so $X_2$ is smooth,
but that $X_n$ is not smooth for $n\geq3$.
%(The proof of
%Proposition~\ref{prop:minimal_dCP} could provide a hint.)
\end{exercise}
\section{Tropical aspects}\label{sec:tropical}
Our third ingredient is a bit of tropical geometry.  The reader
should see \cite{Mi06} or \cite{RST05} for an overview,
and \cite{Ka09} for an advanced introduction to the subject.  The
first ``tropical aspect'' here involves term orders for subvarieties
of the torus.
\subsection{Initial ideals}
Let $I$ be an ideal either in the polynomial ring
$S := \C[x_1,\ldots,x_n]$.  Suppose $u$
is a linear functional (regarded as an element of $\R^n$).
Then $u$ induces an order on monomials.  For $a\in\Z^n$, write
$x^a:=x_1^{a_1}\cdots x_n^{a_n}$, and put $x^a\prec_u x^b$ if and only if
$u(a)<u(b)$.  Then the {\em initial ideal} $\In_u(I)$ is obtained from $I$ by
taking the $\prec_u$-maximal summand of each element of $I$ (allowing ties.)
If $I$ is a homogeneous ideal, then
we may take $u\in N_\R:=\R^n/\R e_{[n]}$.  Impose an equivalence relation on
$N$ by putting $u\sim v$ if and only if
$\In_u(I)=\In_v(I)$.  This is in fact a polyhedral fan, called the
{\em Gr\"obner fan} of $I$: see, for example, \cite{Stu96} for details.

The special case that interests us here is when $I=I(V)$, the defining
ideal of a linear variety $V$.  Then $I$ is generated in degree $1$ by
$V^\perp$, and the
linear elements: i.e., $I$ is the defining ideal of a linear
subvariety $V\subseteq\C^n$ (or, rather, its restriction to $T^n$.)
In that case, $I=(V^\perp)$, and Sturmfels notes in \cite{Stu02} that
the Gr\"obner fan is simply the (outer) normal fan to the matroid polytope.

To see why this should be the case, the key observation is that we can
compute initial ideals using limits of one-parameter subgroups.  Based on
the explicit calculation in Example~\ref{ex:uniform_polytope}, suppose
that $I=(V^\perp)$, and $u\in N$ is some lattice direction, which
again we regard as a one-parameter subgroup of the torus by
letting $u(t)=(t^{u_1},\ldots,t^{u_n})$.  Then, for any vector
$v=\sum_{i=1}^n c_i x_i\in V^\perp$, we have
\[
u(t^{-1})\cdot v= \sum_{i=1}^n c_i t^{-u_i}x_i.
\]
Since $u$ is only defined up to a multiple of the vector $e_{[n]}$, we may
choose our representative so that $-u_i\geq0$ for all $i$, and the
largest coordinate(s) are equal to zero.  Then computing the limit as
$t\to 0$ leaves us with only the $\prec_u$-initial terms of $v$, so
$\lim_{t\to 0}u(t^{-1})\cdot V^\perp=(V_u)^\perp$,
denoting the linear space defined by $\In_u(V)$ by $V_u$.  Since
$(t\cdot V)^\perp=
t^{-1}\cdot V^\perp$ (see Exercise~\ref{ex:inverses}), we see also that
$\lim_{t\to 0}u(t)\cdot V=V_u$.

Using Theorem~\ref{thm:GGMS}, then, $u\sim v$ in the inner normal fan of
$P_{\M}$
if and only if
\begin{eqnarray*}
\lim_{t\to 0}u(t)\cdot V=\lim_{t\to  0}v(t)\cdot V
&\Longleftrightarrow & V_u=V_v,
\end{eqnarray*}
if and only if $u\sim v$ in the Gr\"obner fan of $I$.

Since the cones of a normal fan are in bijection with polytope faces,
the set of all the initial ideals is indexed by the faces of the
matroid polytope.
\begin{proposition}\label{prop:Grob}
If $\M=\M(V)$ and $P_{\M_u}$ is any face of the polytope $P_{\M}$, then
the degeneration matroid $\M_u$ equals
$\M(V_u)$, where $V_u$ is the linear space defined by $\In_u(I(V))$.
\end{proposition}

\subsection{Bergman fans}
Now let us repeat the construction above where, this time, $I$ is an ideal
in the coordinate ring of the torus, $\C[T^n]=\C[x_1^{\pm1},\ldots,x_n^{\pm1}]$.
Assume $I$ is a proper ideal generated in degree $1$.  Then,
from \S\ref{ss:U}, the  zero locus of such an ideal is a
hyperplane complement $U(\A)$ in a linear space $V$, where the matroid
$\M(V)$ contains no loops.  (Conversely, if $i$ is a loop in a linear
matroid $\M(V)$, then $I(V)$ contains a unit, the variable $x_i$,
 which is to say
$V\cap T^n=\emptyset$.)  So by Proposition~\ref{prop:Grob}, {\em the faces of
the matroid polytope corresponding loop-free degenerations are in
bijection with those initial ideals $\In_u(I(V))$ that define
(nonempty) hyperplane complements.}  This motivates the following
definition.

\begin{definition}\label{def:bergman}
For a subvariety $X\subseteq T^n$ defined by an ideal $I$,
its {\em Bergman fan} $\Be(X)$ is the set of
cones $\sigma_u$ in the Gr\"obner fan for which $\In_u(I)$ does not contain a
monomial.
\end{definition}
In the special case where $X=V\cap T^n$, we saw above that the Bergman fan 
depends only on the matroid $\M=\M(V)$, and accordingly we will denote it by
$\Be(\M)$.  This is the set of cones $\sigma_u$
of $\Sigma_{\M}$ for which the degeneration matroid $\M_u$
does not contain a loop.  Theorem~\ref{thm:flacet0} gives
 another characterization, as well.  Recall $\M_u$
contains a loop if and only if the face $u$ lies in the boundary of the standard
simplex, so $\Be(\M)$ consists of (outer) normal cones to the faces of
$\partial(d\cdot\Delta^{n-1})-\partial P_{\M}$.  In order
to visualize it conveniently, we recall that the (outer) normal
fan of a polytope consists of cones over the faces of the polar polytope
$P^*_{\M}$.  That is, $\Be(\M)$ is a cone over a polyhedral subcomplex
of $P^*_{\M}$, which we denote $B_{\M}$.  This is the
{\em Bergman complex} of $\M$.

\begin{example}
Continuing Example~\ref{ex:uniform_polytope}, the defining ideal of $V$ is
\[
I=I(V)=(V^\perp)=(-x_1+x_2+x_3,-x_1-x_2+x_4).
\]
Taking $u=(1,0,0,0)$, for example, and rewriting the subspaces in
\eqref{eq:limit} as kernels, we find $\In_u(I)=(x_1,2x_2+x_3-x_4)$.
By inspection the only vectors $u$ for which $\In_u(I)$ does not
contain a variable are the nonnegative multiples of $(0,1,1,1)$, $(1,0,1,1)$,
$(1,1,0,1)$, and $(1,1,1,0)$, and the Bergman fan consists of the
bold rays in Figure~\ref{fig:octahedron_b}.
\end{example}

\begin{example}
For the matroid of Examples \ref{ex:delA3} and \ref{ex:delA3_again},
the subcomplex of $\partial P_{\M}$ consisting of faces in
$\partial(3\cdot\Delta^4)$ is shown in Figure~\ref{fig:loopy}.
On the other hand, the polar polytope $P^*_{\M}$ has seven vertices,
which lie in the directions $\set{-e_1,e_2,e_3,e_4,e_5,e_{124},e_{135}}$.
The Bergman complex $B_{\M}$
consists of the edges shown in bold in Figure~\ref{fig:bergman_delA3_BM}.
\end{example}
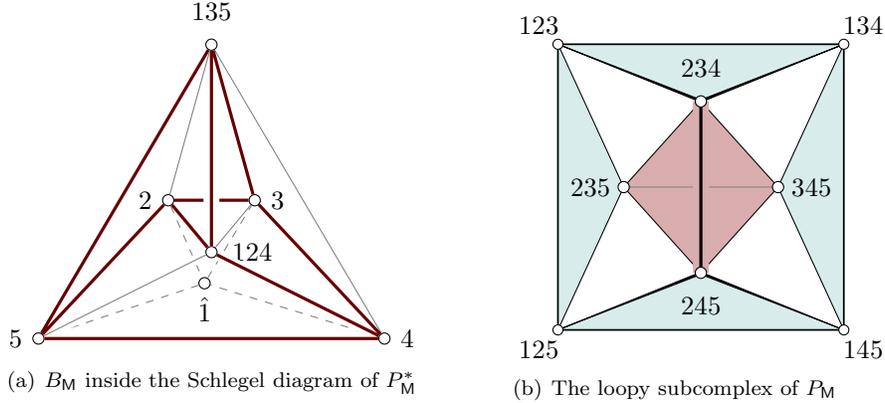
\begin{figure}
\subfigure[$B_{\M}$ inside the Schlegel diagram of
$P^*_{\M}$]{
\begin{tikzpicture}[scale=2.3]
\tikzstyle{every node}=[circle,draw,fill=white,inner sep=1.5pt]
\node[label=above:$135$] (A135) at (0,1.7) {};
\node[label=left:$5$] (A5) at (-1,0) {}; 
\node[label=right:$4$] (A4) at (1,0) {};
\node[label=left:$2$] (A2) at (-0.25,0.8) {};
\node[label=right:$3$] (A3) at (0.25,0.8) {};
\node[label=right:$\;124$] (A124) at (0,0.5) {};
\node[label=below:$\hat{1}$] (A1) at (-0.04,0.32) {};
% order matters because of lines crossing
\draw[style=thin,dashed,color=gray] (A2) -- (A1) -- (A3);  -- ghost vertex
\draw[style=thin,dashed,color=gray] (A4) -- (A1) -- (A5);

\draw[line width=4pt,color=white] (A124) -- (A5);
\draw[line width=4pt,color=white] (A4) -- (A124);

\draw[style=very thick,color=red!40!black] (A2) -- (A3);
\draw[line width=6pt,color=white] (A135) -- (A124);
\draw[style=very thick,color=red!40!black] (A5) -- (A135) -- (A124) -- (A4) -- (A5) -- (A2) -- (A124);
\draw[style=thin,color=gray] (A4) -- (A135) -- (A2);

\draw[style=very thick,color=red!40!black] (A135) -- (A3) -- (A4);
\draw[style=thin,color=gray] (A3) -- (A124) -- (A5);

\end{tikzpicture}
\label{fig:bergman_delA3_BM}
}
\subfigure[The loopy subcomplex of $P_{\M}$]{
% tapered line segments
\pgfdeclaredecoration{triangle}{start}{
  \state{start}[width=0.99\pgfdecoratedinputsegmentremainingdistance,next state=up from center]
  {\pgfpathlineto{\pgfpointorigin}}
  \state{up from center}[next state=do nothing]
  {
    \pgfpathlineto{\pgfqpoint{\pgfdecoratedinputsegmentremainingdistance}{\pgfdecorationsegmentamplitude}}
    \pgfpathlineto{\pgfqpoint{\pgfdecoratedinputsegmentremainingdistance}{-\pgfdecorationsegmentamplitude}}
    \pgfpathlineto{\pgfpointdecoratedpathfirst}
  }
  \state{do nothing}[width=\pgfdecorationsegmentlength,next state=do nothing]{
    \pgfpathlineto{\pgfpointdecoratedinputsegmentfirst}
    \pgfpathmoveto{\pgfpointdecoratedinputsegmentlast}
  }
}

\tikzset{
    triangle path/.style={decoration={triangle,amplitude=#1}, decorate},
    triangle path/.default=1ex}
% end tapered line segments
\lower8pt\hbox{
\begin{tikzpicture}[scale=1.9]
\def\ip{0.6}
\tikzstyle{every node}=[circle,draw,fill=white,opacity=1.0,inner sep=1.5pt]
\coordinate[label=above left:$123$] (A123) at (-1,1) ;
\coordinate[label=above right:$134$] (A134) at (1,1) ;
\coordinate[label=below right:$145$] (A145) at (1,-1);
\coordinate[label=below left:$125$] (A125) at (-1,-1);

\coordinate (A235) at (-0.9*\ip,0);
\coordinate (A234) at (0,\ip);
\coordinate (A345) at (0.9*\ip,0);
\coordinate (A245) at (0,-\ip);

%\node[label=left:$235$] (A235) at (-\ip,0) {};
%\node[label=above:$234$] (A234) at (0,\ip) {};
%\node[label=right:$345$] (A345) at (\ip,0) {};
%\node[label=below:$245$] (A245) at (0,-\ip) {};

\draw[fill=red!20!white!85!black] (A235) -- (A234) -- (A345) -- (A245) -- (A235) -- cycle;
\begin{scope}[fill opacity=0.5, fill=green!50!blue!30!white]
\fill (A123) -- (A134) -- (A234) -- cycle;
\fill (A123) -- (A235) -- (A125) -- cycle;
\fill (A125) -- (A245) -- (A145) -- cycle;
\fill (A134) -- (A345) -- (A145) -- cycle;
\end{scope}
\draw[line width=0.65pt] (A123) -- (A134) -- (A145) -- (A125) -- (A123);

\draw[style=very thin,gray] (A235) -- (A345);  % back of tetrahedron
\draw[line width=6pt,color=red!20!white!85!black] (A234) -- (A245);
\draw[style=very thick] (A234) -- (A245);  % front

% remaining lines come out of canvas

\def\tipwidth{0.4pt}
\draw[fill=black,triangle path=\tipwidth] (A123) -- (A234);
\draw[fill=black,triangle path=\tipwidth] (A134) -- (A234);

%\draw[fill=black,triangle path=\tipwidth] (A235) -- (A123);
%\draw[fill=black,triangle path=\tipwidth] (A235) -- (A125);

\draw[fill=black,triangle path=\tipwidth] (A125) -- (A245);
\draw[fill=black,triangle path=\tipwidth] (A145) -- (A245);

%\draw[fill=black,triangle path=\tipwidth] (A345) -- (A145);
%\draw[fill=black,triangle path=\tipwidth] (A345) -- (A134);

\draw (A145) -- (A345) -- (A134);
\draw (A123) -- (A235) -- (A125);

\foreach \v in {(A123),(A134),(A234),(A235),(A125),(A245),(A145),(A345)} {
  \draw[fill=white] \v circle (1pt);
}

\node[label=left:$235$] at (A235) {};
\node[label=above:$234$] at (A234) {};
\node[label=right:$345$] at (A345) {};
\node[label=below:$245$] at (A245) {};

\end{tikzpicture}
} % end hbox
\label{fig:loopy}
} % end subfigure
\caption{A Bergman complex and its dual, Example~\ref{ex:delA3}}
\label{fig:bergman_delA3}
\end{figure}

\begin{remark}\label{rem:degree}
Ardila and Klivans~\cite{AK06} show that the Bergman complex $B_{\M}$
is homeomorphic to the order complex of the open interval $(\bottom,\top)$
in $L(\M)$.  The latter is known to be homeomorphic to a wedge of
$\mu$ many $(d-2)$-spheres, where the number $\mu$
is, up to sign, the value of the
M\"obius function over the lattice: $\mu=(-1)^{d}\mu_{L(\M)}(\bottom,\top)$.
Hacking finds an interesting generalization of this in  \cite{Ha08}.
The number $\mu$ is also the
top Betti number of $\P U(\A)$, if $\M$ is the matroid of a complex arrangement
$\A$.

By Alexander duality, the subcomplex of $P_{\M}$ consisting of faces indexed
by matroids with loops also has the
homology of a wedge of $\mu$ spheres of dimension $n-d-1$.  In the examples
shown in Figures~\ref{fig:octahedron}, \ref{fig:loopy}, it can be
seen that this ``loopy subcomplex'' is homotopic
to a wedge of spheres.  It would be interesting to know a direct proof.
\end{remark}

\subsection{Amoebas}
For a subset $X\subseteq T^n$, its (classical) {\em amoeba} is defined to be
the set
\[
A_t(X):=\set{(\log_t\abs{x_1},\log_t\abs{x_2},\ldots,\log_t\abs{x_n})\colon
x\in X}\subseteq \R^n,
\]
where $t>1$ is the base of the logarithm.
If $X$ is closed under the diagonal action of $\C^*$, then $A_t(X)$ is
invariant under translation by the vector $e_{[n]}$, so
$A_t(X)$ can be taken to be a subset of $\R^n/\R e_{[n]}$.

The {\em logarithmic limit set}, denoted $\Log(X)$, is the limit (in the
Hausdorff metric) of the amoebas $A_t(X)$ as $t\to\infty$.
%If $X$ is an irreducible complex
%algebraic variety, this agrees with an algebraic construction and
%is also called the {\em tropicalization} of $X$.  The key here is
%a description using Gr\"obner bases, for which we refer to \cite{Stu02}.
If $X$ is an algebraic subvariety of the torus, then a foundational
result in tropical geometry from \cite[\S 9]{Stu02} says that points in
$\Log(X)$ are indexed by initial ideals defining nonempty varieties.
\begin{theorem}[\cite{Stu02}]\label{thm:tropical}
For a subvariety $X$ of $\PT^{n-1}$, the set
$-\Log(X)$ equals the support of the Bergman fan, $\abs{\Be(X)}$.
In particular, for a hyperplane complement $\P U\subseteq \PT^{n-1}$,
we have $-\Log(\P U)=\abs{\Be_{\M(V)}}$.
\end{theorem}

The tropicalization  of $X$ consists of the set $\abs{\Be(X)}$, together with
some additional (integer) data which we will ignore here.  
In the case where $X$ is linear, it
is clear what information about $X$ is preserved in passing to
its tropicalization: while an amoeba $A_t(X)$ depends on the equations defining
$X$, the logarithmic limit set depends only on the underlying matroid.

\begin{figure}
\begin{tikzpicture}
\begin{scope}
\clip (0.1,0.1) circle (1.2);
\draw[draw=white,fill=cyan!10!white] (0.144,-1.2) .. controls (0.25,0) .. (1.4,1.2) --
(1.2,1.4) .. controls (0,0.25) .. (-1.2,0.144) -- (-1.2,-0.144)
 .. controls (-0.2,-0.2) ..( -0.144,-1);
\draw[style=very thick] (0,0) -- (0,-1.2);
\draw[style=very thick] (0,0) -- (-1.2,0);
\draw[style=very thick] (0,0) -- (1.35,1.35);
\draw (0.144,-1.2) .. controls (0.25,0) .. (1.4,1.25);
\draw (-0.144,-1.2) .. controls (-0.2,-0.2) ..( -1.2,-0.144);
\draw (-1.2,0.144) .. controls (0,0.25) .. (1.25,1.4);
\end{scope}
\draw[dashed,->] (0,0) -- (0,1.4);
\draw[dashed,->] (0,0) -- (1.4,0);
\draw[dashed,->] (0,0) -- (-1,-1);
\end{tikzpicture}
\caption{Amoeba and logarithmic limit set for $V$ with $\M(V)=U_{2,3}$}
\label{fig:amoeba}
\end{figure}
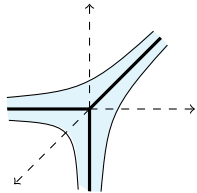

\begin{example}
Consider a plane $V$ with $\M(V)=U_{2,3}$.  Then
\[
A_t(\P U)=\set{(\log_t\abs{z},\log_t\abs{w},\log_t\abs{z-w})\colon [z:w]\in
\P U(\A)}\subseteq\R^3/\R(1,1,1)
\]
An amoeba and logarithmic limit set for $\P U$ are shown in
Figure~\ref{fig:amoeba}.
\end{example}

\section{Compactifications}\label{sec:cpct}
If $\A$ is an arrangement,
we will say $Y$ is a {\em compactification} of $\A$ if
$Y$ is a complete complex variety and $\P U(\A)$ is a dense open subset
of $Y$.  Recall that a Zariski open subset is dense if and only if it
is dense in the complex topology, and complete varieties are compact
in the complex topology.
Of course, one way to produce a compactification
is to take the closure of the embedding of $\P U(\A)$ in some complete
variety, and we will do this here.  Since arrangement complements
sit naturally as subvarieties of tori, it makes sense to consider
their closures in various toric varieties.  We will give some examples
before considering discussing the general theory in the next section.

\subsection{The reciprocal plane}
Of course, $\P V$ is a compactification of $\A$, obtained from embedding
the complement in $\P^{n-1}$.  A
more interesting example with the same ambient space
can be obtained by letting $i\colon \PT^n\to \PT^n$
denote the inverse map in the torus, and letting
\begin{equation}\label{eq:recip}
Y(\A)=\overline{i(\P U(\A))}\subseteq \P^{n-1}.
\end{equation}
(Regarded as a rational map, $i\colon \P^{n-1}\dashrightarrow\P^{n-1}$
is called the standard Cremona transformation.)  This construction
has been studied in a number of papers: see, for example,
\cite{PS06,ST09,Loo03,HT03,SSV11}.  Recently it appeared in the work of
Huh and Katz on the log-concavity of the coefficients of the
characteristic polynomial (in arbitrary characteristic): see \cite{HK11}
and \cite{Le11}.

In particular, the homogeneous coordinate ring of $Y(\A)$ equals
$\C[1/f_i\colon 1\leq i\leq n]$.  To describe this as a quotient of a
polynomial ring, continue to
let $S:=\C[x_1,\ldots,x_n]$ denote the homogeneous coordinate ring of
$\P^{n-1}$.  The inclusion $Y(\A)\subseteq \P^{n-1}$
induces a surjective homomorphism
\begin{equation}\label{eq:OTpres}
S\to \C[1/f_i\colon 1\leq i\leq n]
\end{equation}
sending $x_i$ to $1/f_i$, for $1\leq i\leq n$.  Let $I(\A)$ denote the
kernel of the map \eqref{eq:OTpres}.

To describe generators of $I(\A)$, let $\supp{c}=\set{i\in[n]\colon c_i\neq 0}$
for $c\in\C^n$, and let
%
%Then generators of $I(\A)$ are indexed by circuits of the matroid $\M(V)$.
%For any $c\in\C^n$, let
%If $\supp{c}$ is a circuit, % MARK: still to define above
%then
%{\em circuits} of $\A$ are those $c\in\C^{n}$ for which
%$\sum_{i=1}^n c_i f_i=0$ and $\supp{c}$ is minimal.
%If $c$ is a circuit, the element
\begin{equation}\label{eq:OTrels}
r_c:=\sum_{i\in \supp{c}}c_i \prod_{j\in\supp{c}-\set{i}} x_j.
%y_1\cdots\widehat{y_i}\cdots y_k
\end{equation}
If $c\in V^\perp$, then $\sum_{i=1}^n c_i f_i=0$, so $r_c\in I(\A)$.
A set of generators for $I(\A)$ is given by selecting those $r_c$ for which
$c\in V^\perp$ and $\supp{c}$ is minimal (the {\em circuits} of $\A$).
It is known that $S/I(\A)$ is Cohen-Macaulay,
and that it has a Gr\"obner basis indexed by the broken circuit
complex~\cite{PS06}.  Questions about syzygies of $I(\A)$, such as
degrees of minimal sets of generators and Castelnuovo-Mumford regularity,
seem to be difficult in general: see \cite{ST09}.
Sanyal, Sturmfels and Vinzant give a matroidal condition
that characterizes those subsets of the equations \eqref{eq:OTrels}
determine $Y(\A)$ set-theoretically in \cite[Prop.~21]{SSV11}.

%In fact, if ${\mathcal S}$ is a subset of the circuits of $\M(V)$, then
%Sanyal, Sturmfels and Vinzant show \cite[Prop.~21]{SSV12} that
%\[
%Y(\A)=V(\set{r_c\colon c\in {\mathcal S}})  % careful, it's proj
%\]
%if and only if $\mathcal S$ exposes every non-flat.  This means that, if
%$J\subset[n]$ is not a flat of $\M(V)$, there exists a circuit
%$c\in {\mathcal S}$ for which $\abs{J^c\cap \supp{c}}=1$.  % keep? omit?

\begin{example}\label{ex:uniform2}
If $d=\dim(V)=2$ and any two coordinates of $f\colon V\to \C^n$ are
linearly independent, then $\M(V)\cong U_{2,n}$ (Example~\ref{ex:uniform}).
The circuits of the matroid consist of all three-element subsets of $[n]$,
and the ideal $I(\A)$ is generated by quadrics.  By general theory (or
just hands-on linear algebra), $n-1\choose 2$ quadrics are required.
Since $\P V$ has codimension $n-2$ in $\P^{n-1}$, so does $Y(\A)$.  Hence
the depth (and projective dimension) of $I(\A)$ are equal to $n-2$.
%It turns out
%that, in this case, $I(\A)$ has a Hilbert-Burch resolution: see \cite{ST09}.

More generally, if $\M(V)\cong U_{d,n}$, a uniform matroid with arbitrary
parameters, then again the syzygies of $I(\A)$ can be understood completely.
$I(\A)$ is generated by $n-1\choose k$ generators of
degree $d$, and $I(\A)$ has a linear resolution which is an
Eagon-Northcott complex: see \cite[\S 4.4]{DGT12}.
\end{example}
\subsection{Visible contours}\label{ss:vc}

The reciprocal plane lacks some features that one normally expects from
a compactification: in particular, $Y(\A)$ is not smooth, in general.
(The singular locus is described in  \cite{SSV11}.)  The following
construction is an improvement. 

Recall that,
provided the matroid $\M(V)$ is connected, the torus orbit of $V$ in $\Gr_{d,n}$
given by $a_V(t)=t\cdot V$ was isomorphic to $\PT^{n-1}$, and we saw
in the previous section that the orbit closure was a toric variety
$X_{\M}$ given, abstractly, by the matroid $\M=\M(V)$ (Theorem~\ref{thm:GGMS}).

Since the arrangement complement $\P U(\A)$ is given as a subspace of
$\PT^{n-1}$, it makes sense to consider its closure inside
$\Gr_{d,n}$.  Instead of doing this directly, though, we want to use the
reciprocal embedding instead, and define
\begin{equation}\label{eq:Xvc}
Y_{\vc}(\A)=\overline{a_V\circ i(\P U(\A))}\subseteq \Gr_{d,n}.
\end{equation}
This is the {\em visible contours} compactification from \cite{Ka93}.  It 
is also called the {\em tropical} compactification in \cite{FS05}; 
however, we will avoid the term here, since there are various 
tropical compactifications in the sense of Tevelev \cite{Te07}.
We
leave it as an exercise to the reader to find the appropriate modifications
in the case where $\M(V)$ has more than one connected component.

To see why the inverse map should appear here, consider replacing $V$ by
another subspace in the same orbit, $V':=t\cdot V$,
for some $t\in \PT^{n-1}$.  The hyperplane complements are related by
$\P U':=t\cdot \P U$.  Then
\begin{eqnarray*}
a_{V'}\circ i(\P U') &=& a_{tV}\circ i(t\cdot \P U)\\
&=& a_{V}\circ t i(t\cdot \P U)\\
&=& a_V\circ i(\P U),
\end{eqnarray*}
so the construction of $Y_{\vc}(\A)$, as a subvariety of the Grassmanian, 
is independent of our choice of subspace from the orbit $\PT^{n-1}\cdot V$.

We leave some straightforward assertions about the inverse map as an exercise:
\begin{exercise}\label{ex:inverses}\leavevmode
\begin{itemize}
\item For $t\in \PT^{n-1}$ and $V\in \Gr_{d,n}$, we have $(i(t)\cdot V)^\perp=
t\cdot V^{\perp}$.
\item If $X$ is a toric variety with torus $T$, define a toric variety
$X^{-1}$ whose underlying variety is also $X$, but the action of $T$
on $X^{-1}$ is constructed from the action on $X$ by letting $t\cdot x=i(t)x$.
Show that $\Sigma_{X^{-1}}=-\Sigma_X$.  If $X$ is projective, show that
$P_{X^{-1}}=-P_X$.  Decide when $X$ and $X^{-1}$ are isomorphic as
toric varieties.
\item For any matroid $\M$, check the weight polytope of $X_{\M}^{-1}$ is
$P_{\M^*}$ (see Exercise~\ref{ex:dual_polytope}.)
\end{itemize}
\end{exercise}

Of course we could also define $Y_{\vc}(\A)$ as the closure of the
hyperplane complement in the toric variety $X_{\M}^{-1}$, since
\[
\P U(\A)\hookrightarrow X_{\M}^{-1}\subseteq \Gr_{d,n}.
\]
\begin{comment}
Using the Pl\"ucker embedding, we can write the homogeneous coordinate
ring for $Y_{\vc}(\A)$ explicitly:
\begin{proposition}
For any arrangement $\A$, the homogeneous coordinate ring of $Y_{\vc}(\A)$
in $\P^{\abs{\B}-1}$ equals
\[
\C[(\prod_{i\in B} f_i)^{-1}\colon B\in \B],
\]
where $\B$ is the set of bases of $\M(\A)$.
\end{proposition}
\begin{proof}
By putting $-^{-1}$ into the proof of Theorem~\ref{thm:GGMS}, we saw that the homogeneous coordinate
ring of $X_{\M}^{-1}$
% MARK
\end{proof}
\end{comment}
The hyperplane complement has codimension $n-d$ in $X_{\M}^{-1}$, and its
closure only intersects some of the torus orbits.  Those orbits have a nice
description.

\begin{theorem}[\cite{FS05}]\label{thm:bergman}
For all $\sigma\in \Sigma_{\M}$, we have
$O(\sigma)\cap Y_{\vc}(\A)\neq\emptyset$ if and only if $\sigma\in \Be(\M)$.
That is,
the smallest toric variety in $X_{\M}^{-1}$ that contains $Y_{\vc}(\A)$
is $X_{\Be(\M)}$.
\end{theorem}

We note that $\Be(\M)$ is not a complete fan, so our minimal compactifying
toric variety $X_{\Be(\M)}$ is not, itself, compact.  
In \cite[Thm.~1.5]{Te07}, Tevelev shows that linear spaces are {\em sch\"on},
which means in particular that the restriction of the multiplication map
$
\PT^{n-1}\times Y_{\vc}(\A)\to X_{\Be(\M)}
$
is smooth.  It follows that $Y_{\vc}(\A)$ is smooth if and only if 
$X_{\Be(\M)}$ is smooth, and in general, $Y_{\vc}(\A)$ has at 
most toroidal singularities.
Examples show that
the Bergman fan need not be simplicial, so the visible contours compactification
is not, in general, smooth: for a family of examples, see \cite{DD12}.
Here is one from the literature.
\begin{example}
Let $\A$ be a
hyperplane arrangement consisting of the six face planes of a cube
in $\P^3$: for example, let $V$ be the image in $\C^4$ of 
\[
f(z,u,v,w)=(z+u,z-u,z+v,z-v,z+w,z-w).
\]
Then $\B_{\M(\A)}$ is a nonsimplicial $2$-complex, described
explicitly in \cite[Ex.~2.8]{FS05}.
\end{example}

\begin{exercise}
Note that $\Gr_{1,n}=\P^{n-1}$, and find a family of compactifications that
interpolates between the reciprocal plane \eqref{eq:recip} and the
visible contours compactification \eqref{eq:Xvc}.
\end{exercise}

\begin{example}\label{ex:uniform3}
For $\M=U_{2,4}$, we saw that
$\Be(\M)$ consists of the four rays
shown in bold in Figure \ref{fig:octahedron_b}, and Theorem~\ref{thm:bergman}
can be verified directly.
For any $t\in \PT^3$, the plane $i(t)\cdot V$ is the row space of
\begin{equation}\label{eq:tA_U24}
\begin{pmatrix}
t^{-1}_1 & 0 & t^{-1}_3 & t^{-1}_4\\
0 & t^{-1}_2 & -t^{-1}_3 & t^{-1}_4
\end{pmatrix}.
\end{equation}
%Taking $2\times2$ minors, its image in $\P^5$ is
%\begin{equation}\label{eq:U24minors}
%[t_1^{-1}t_2^{-1}:-t_1^{-1}t_3^{-1}:t_1^{-1}t_4^{-1}:-t_2^{-1}t_3^{-1}:
%-t_2^{-1}t_4^{-1}:2t_3^{-1}t_4^{-1}].
%\end{equation}
%which is the set of solutions to
%\begin{align}\label{eq:tV_U24}
%-t_1x_1+t_2x_2+t_3x_3+\phantom{t_4x_4} &=0\\
%-t_1x_1-t_2x_2\phantom{+t_4x_4}+t_4x_4 &=0.\nonumber
%\end{align}
On the other hand, a point
$[z:w]\in\P^1$ is in $\P U$ provided that $z,w\neq0$ and $z\neq
\pm w$.  Its image under $a_V\circ i$, using \eqref{eq:tA_U24},
is the set of solutions to
\begin{align*}
-zx_1 +wx_2+(z-w)x_3\phantom{+x_4} &=0\\
-zx_1 -wx_2\phantom{+x_3}+(z+w)x_4 &=0.
\end{align*}
In this case, we simply have $Y_{\vc}(\A)\cong\P^1$, since the
boundary points are obtained by letting $z=0$, $w=0$, $z=w$, or $z=-w$.
In each case, the linear space of solutions is not contained in a
coordinate hyperplane, so its matroid has no loop, and it
lies in a torus orbit indexed by a shaded face of $P_{\M}$ in
Figure~\ref{fig:octahedron}.
\end{example}

\subsection{Maximum likelihood, master functions and critical points}
Before moving to the next construction, we digress slightly
in order to mention another,
closely related variety.  For a subspace $V$, define
\begin{equation}\label{eq:gammaV}
\Gamma_V=\set{(x,y)\in \PT^{n-1}\times\C^n\colon y\in x\cdot
V^{\perp}},
\end{equation}
noting again that $x\cdot V^\perp=(i(x)\cdot V)^\perp$.
Pulling back the projection onto the first factor along the inclusion
of the hyperplane complement gives a new variety $\crit(\A)$, the
{\em variety of critical points}:
\[
\xymatrix{
\crit(\A)\ar@{.>}[r]\ar@{.>}[d] & \Gamma_V\ar[d]\\
\P U(\A)\ar[r]^{f} & \PT^{n-1}.
}
\]
To justify the name, we should restrict our attention to  lattice points
$m\in\Z^n$.  Then $(x,m)\in\crit(\A)$
if and only if $x$ is a critical point of the master function
$\prod_{i=1}^n f_i^{m_i}$ (see Remark~\ref{rem:veryaffine}).
The second factor of \eqref{eq:gammaV} is stable under the diagonal
action of $\C^\times$: let $\X(\A)$ denote the closure of the quotient
of $\crit(\A)$ in $\P^{n-1}\times\P^{n-1}$: this is compact, but not
smooth.  It is known to be Cohen-Macaulay for all arrangements $\A$,
but not, in general, arithmetically Cohen-Macaulay (\cite{CDFV11}).

If we let $\P U^\perp=\P V^\perp\cap\, \PT^{n-1}$, the complement of the dual
arrangement, then $\crit(\A)/\C^\times$ contains a simple-looking
dense, open subset:
\begin{align*}
(\crit(\A)/\C^\times)\cap T&=\set{(x,xy)\in T\colon
x\in \P U,\,y\in \P U^\perp},\\
&\cong \P U\times \P U^\perp,
\end{align*}
where $T:=\PT^{n-1}\times\PT^{n-1}$.  With this in mind, we see that
the compactification $\X(\A)$ is the closure of $\P U\times 
\P U^\perp$ in a toric variety isomorphic
to $\P^{n-1}\times\P^{n-1}$ with a nonstandard torus action.  
The automorphism $T\to T$ given by 
$(x,y)\mapsto(x,xy)$ induces a linear automorphism of the character
lattice $M=\Z^{n-1}\times \Z^{n-1}$, and the weight polytope of the
toric variety is just the image of $\Delta^{n-1}\times \Delta^{n-1}$
under this map.  We leave the details to the reader.

The critical points of a master function are interesting from various
points of view, including algebraic statistics, mathematical physics,
and the theory of hyperplane arrangements: see, in particular,
\cite{CHKS06,DGS12}.  Some remarkable results about critical point
varieties about more general very affine varieties appears in \cite{Huh12}.

By regarding the critical point variety as a compactification
of a product of projective linear spaces, we see some symmetry
between the underlying matroid and its dual that is not immediately apparent
from the perspective of critical points of master functions.  In particular,
$\X(\A(V))$ and $\X(\A(V^\perp))$ are both compactifications of
$\P U\times \P U^\perp$, while their 
respective ambient toric varieties differ by
the involution exchanging the factors.

\subsection{Wonderful models}\label{ss:wonderful}
In 1995, De Concini and Procesi~\cite{deCP95} studied a family of
compactifications obtained by iteratively blowing up the projective space
$\P V$ along (proper transforms of) linear subspaces, in increasing order
of dimension.  Their compactification can be obtained in several different
ways, and it has some very desirable properties.
In particular, the boundary of $\P U(\A)$ is well-behaved:
it is a union of hypersurface components, each one isomorphic to
$\P^{d-1}$.  These components intersect with normal crossings,
which is to say that
the neighbourhood of a boundary point looks locally like an
intersection of coordinate hyperplanes.

The number of boundary components and the way in which they intersect
depends on some interesting combinatorics, which we will now describe.
A good expository resource is provided by \cite{Fe05}.  A detailed
analysis and abstraction of the relevant combinatorics appears in 
\cite{FK04}.

\begin{definition}\label{def:building}
Let $L$ be a partially ordered set with unique minimal element $\bottom$.
A subset $G\subseteq L-\set{\bottom}$ is a {\em building set} if, for every
$X\in L$, we have an order-isomorphism
\begin{equation}\label{eq:building}
[\bottom,X]\cong\prod_{Y\in \max(G\cap[\bottom,X])}[\bottom,Y],
\end{equation}
where, for a set $S\subseteq L$, the notation
``$\max\,{S}$'' denotes the subset of maximal elements.
\end{definition}

\begin{lemma}[Prop.~2.5(1), \cite{FK04}]\label{lem:FK2}
Suppose $G$ is a building set, $Y\in L$, and $X\in G$.  If $X\leq Y$,
then $X\leq Y_i$ for a unique element $Y_i\in\max(G\cap[\bottom,Y])$.
\end{lemma}

The wonderful models are parameterized by building sets in the lattice of
flats, $L(\M)$.
\begin{example}
Tautologically, $G=L_{\geq1}(\M)$
is a building set for $L(\M)$.  On the other hand,
it is not hard to see that if $G$ is a building set, then it must be the
case that $G\supseteq L_\irr(\M)$.  By the discussion in \S\ref{ss:restrs},
if $X$ is not irreducible, then $L(\M|X)=[\bottom,X]$ decomposes as a
product of lower intervals in $L(\M)$, from which it follows that
$L_\irr(\M)$ is, itself, a building set.
So we see that $G_{\min} := L_\irr(\M)$ and $G_{\max} := L_{\geq1}(\M)$
are the unique maximal and minimal
building sets, respectively, and building sets themselves form a poset
under inclusion.
See, for example, \cite{GS12}, where the authors consider families of
building sets.
\end{example}

Let $\A$ be a hyperplane arrangement in $V\subseteq\C^n$, and let
$G\subseteq L(\M(V))$ be a building set.
For each $X\in G$, the coordinate projection
$\C^n\to \C^X$ from \S\ref{ss:restrs} induces a rational map
$p_X\colon\P^{n-1}\dashrightarrow \P^{\abs{X}-1}$ which is regular
(i.e., defined) on the torus $\PT^{n-1}$.
Let
\begin{equation}\label{eq:diagonalmap}
p\colon \P^{n-1}\dashrightarrow\prod_{X\in G} \P^{\abs{X}-1}
\end{equation}
be the map whose $X$th coordinate is $p_X$.  Again, this is regular on
$\PT^{n-1}$, so its restriction to $\P U(\A)$ is as well.  By
definition, the image of $V$ under coordinate projection is $V_X$, so
we may factor the restriction of $p$ as follows:
\begin{equation}\label{eq:triangle}
\xymatrix@R-25pt{
& \prod\limits_{X\in G}\P^{\abs{X}-1}\\
\P U(\A)\ar[ur]^-{p}\ar[dr] & \\
& \prod\limits_{X\in G}\P V_X\ar[uu]
}
\end{equation}
Since $\M$ is connected, the maximal flat $[n]$ is in $G$, so
$\P^{n-1}$ and $\P V$ are factors in the top and bottom products, respectively.
Then the maps in \eqref{eq:triangle} are all injective, so
the next definition makes sense.

\begin{definition}\label{def:wonderful1}
The De Concini-Procesi wonderful compactification of
$\A$ with building set $G$ is
\begin{equation}\label{eq:XdCP}
Y_{\dCP}(\A,G)=\overline{p(\P U(\A))}\subseteq
%\prod_{\mathclap{X\in G\cup\set{\bottom}}}  \P V_X.  % need mathtools
\prod_{X\in G}  \P V_X.
\end{equation}
\end{definition}

The boundary components in $Y_{\dCP}(\A,G)$ are indexed by the
building set $G$.  It remains to say which boundary components intersect,
and we will see that this depends only on the matroid of $\A$ and the
choice of building set.

Note that a building
set $G\subseteq L(\M)$ necessarily contains all singleton flats $\set{i}$,
since one-element matroids are connected.
\begin{definition}\label{def:nested}
Let $G\subseteq L$ be a building set in a lattice $L$.
A subset $S\subseteq G$ is a {\em nested set} for $G$ provided that, if
$X_1,\ldots,X_k$ are pairwise incomparable elements of $S$, and $k\geq2$,
then the join $X_1\vee\cdots\vee X_k\not\in G$.  Let $\Ne(G)$ denote the
set of all nested sets for $G$.
\end{definition}
\begin{example}
Certainly if $S\subseteq G$ is a chain, it is nested.
If $G=G_{\max}$, the maximal building set, then a nested set cannot
have incomparable elements, so in this case $\Ne(G)$ is simply the set of
chains in $L(\M)$.  On the other hand, if $G=G_{\min}$, then
a nested set $S$ can contain incomparable elements, provided that their
join is not irreducible.
\end{example}
\begin{exercise} % maybe
For the Boolean arrangement $\CC_n$, the maximal building set
consists of all nonempty subsets of $[n]$.  Check that maximal nested
sets for $G_{\max}$ are indexed by permutations of $[n]$.
On the other hand, $G_{\min}$ consists of just the rank-$1$ flats
(hyperplanes).  Describe the nested sets in this case.
\end{exercise}
\begin{exercise}\label{ex:NScone}
If $\M$ is connected, the maximal flat $[n]$ is in every building set.
Let $\Ne_0(G)$ denote the nested sets that do not contain $[n]$.
Check that $\Ne(G)$ and $\Ne_0(G)$ are simplicial complexes
on the building set $G$ and $G-\set{[n]}$, respectively.  Show that
$[n]$ is in every maximal nested set, so 
$\Ne(G)$ is a cone over $\Ne_0(G)$ at the vertex $[n]$.
\end{exercise}

These definitions were made to keep track of the incidence structure in
the boundary of the compactification.

\begin{theorem}[\S 4.2, \cite{deCP95}]
The boundary $Y_{\dCP}(\A,G)-p(\P U(\A))$ is a union of divisors
$\set{D_X\colon X\in G-\set{[n]}}$.  For a subset
$S\subseteq G-\set{[n]}$, 
the intersection $\bigcap_{X\in S}D_X$ is nonempty if and
only if $S$ is a nested set for $G$, in which case the intersection is
transversal and irreducible.
\end{theorem}

It should be mentioned that De Concini and Procesi's construction
is more general, allowing an arbitrary union of linear subspaces
(over any infinite field) in place of an arrangement of hyperplanes;
however, it is worth singling out this special case, since it
arises most often, and the theory is somewhat simpler.  Here we give
the projective version of their construction: an affine variation gives 
a closure of $U(\A)$ with an additional divisor indexed by the maximal
flat.

\begin{example}\label{ex:delA3_nested}
Consider the matroid of Example~\ref{ex:delA3} again.  In this case,
we have
$G_{\min}=\set{1,2,\ldots,5,124,135,12345}$.  A realization of 
the complex $\Ne_0(G_{\min})$ is given by the bold edges shown in
Figure~\ref{fig:irr_refinement}.
\end{example}

Here are some facts about nested sets which we will use in the next section.
Although nested sets need not be chains, in general, they behave like
chains in the following ways.
%\begin{lemma}[Prop.~2.5(1), \cite{FK04}]\label{lem:unique}
%For any $Y\in L(\M)$, and $X\in G$ for which $\bottom\neq X\leq Y$,
%there exists a unique element $Y'$ in $\max(G\cap[\bottom,Y])$ for
%which $X\leq Y'$.
%\end{lemma}
\begin{lemma}[Prop.~2.8(2), \cite{FK04}]\label{lem:FK1}
Suppose $Y_1,\ldots,Y_k$ are pairwise incomparable elements in a nested
set for a building set $G$, and $Y=Y_1\vee\cdots\vee Y_k$.  Then
the elements of $\max(G\cap[\bottom,Y])$ are precisely $\set{Y_1,\ldots,Y_k}$.
\end{lemma}

\begin{lemma}\label{lem:minimal}
Let $G\subseteq L(\M)$ be a building set, and $S\in \Ne(G)$ a nested set
for $G$.  For each $i\in [n]$, let $S_i=\set{X\in S\colon i\in X}$.
Then, for any $X\in L(\M)$, the collection of nested sets
$
\set{S_i\colon i\in X}
$
has a unique minimal element under inclusion; i.e., there exists some
$i_0\in X$ for which $S_{i_0}\subseteq S_i$, for all $i\in X$.
\end{lemma}
\begin{proof}
We can assume $X\neq\bottom$, for which the result is trivial.
Suppose the claim is false: that is, for each $i\in X$, there is a
flat $Y_i\in S_i-\bigcap_{j\in X}S_j$.  Consider the set
$\set{Y_i\colon i\in X}\subseteq L(\M)$, and let
$\set{Y_{i_1},Y_{i_2},\ldots,Y_{i_k}}$ denote
its subset of maximal elements (a pairwise incomparable set).
We must have $k\geq2$, since otherwise
the flats $\set{Y_i}$ would form a chain, in which case the maximal element
$Y_{i_1}\in S_j$ for all $j$, a contradiction.

By the nested set property, the join
\[
Y:=Y_{i_1}\vee\cdots\vee Y_{i_k}
\]
is not an element of $G$.  Since $Y=\bigvee_{i\in X} Y_i$ as well, and
$i\in Y_i$ for each $i\in X$, we see $X\leq Y$.  Since $G$ is a building
set, and $X\in G$, it must be the case that $X\leq Y_{i_j}$ for some $j$,
by Lemma~\ref{lem:FK1}.  But then $Y_{i_j}\in S_i$ for all $i$, a contradiction.
\end{proof}
\begin{lemma}\label{lem:Smin}
The nonempty sets $S_i$ are chains.
For every $X\in S$, there is some $i$ for which $X=\min S_i$.
\end{lemma}
\begin{proof}
Suppose $X$ and $Y$ are incomparable elements of $S$, and $X\in S_i$ for
some $i$.
The flat $\set{i}$ is irreducible, hence in $G$.
Since $\set{i}\leq X$, we must have $\set{i}\not\leq Y$,
by Lemma~\ref{lem:FK2}.  That is, which is to say $Y\not\in S_i$.

To check the second claim, suppose some $X\in S$ is not minimal in
any chain $S_i$.  That is, for all $i\in X$, there exists some $Y_i\in S_i$
for which $Y_i<X$.  Then $X\leq \bigwedge_{i\in X}Y_i$, so $X\leq Y_i$ for
some $i$ by Lemma~\ref{lem:FK1}, a contradiction.
\end{proof}
Clearly if $X\in S_i$, then $Y\in S_i$ too, for all $Y\in S$ with $Y\geq X$.
So we see, in particular,
that, if we draw edges between comparable elements of a nested set $S$,
the graph we obtain is a forest, with the leaves at the bottom.

For any nested set $S$, define an equivalence
relation on the set $\bigvee_{X\in S}X\subseteq[n]$ by letting $i\sim_S j$
if and only if $S_i=S_j\neq\emptyset$.  
The corresponding partition has $\abs{S}$ blocks.  
%If $S$ is a maximal nested set, each
%$S_i\neq\emptyset$, so the partition has at most $d=\abs{S}$ blocks.

\begin{lemma}\label{lem:partition}
Let $S\in \Ne(G)$ be a nested set.  
Let $S=\set{X_1,\ldots,X_r}$ be any linear extension.
%: that is, for all $1\leq i,j\leq d$, if $X_i<X_j$, then $i<j$.
The sets
\begin{equation}\label{eq:difference}
E_k :=\bigvee_{i=1}^{k} X_i-\bigvee_{i=1}^{k-1} X_i\quad\subseteq [n]
\end{equation}
for $1\leq k\leq r$ are all nonempty, and they are the blocks of the
partition $\sim_S$.
\end{lemma}
\begin{proof}
First, we check that $E_k$ is nonempty for all
$k$.  If not, then $\bigvee_{i=1}^k X_i=\bigvee_{i=1}^{k-1} X_i$, which
means $X_k\leq \bigvee_{i=1}^{k-1}X_i$.  By Lemma~\ref{lem:FK1}, then
$X_k\leq X_i$ for some $i\leq k-1$, which would contradict having
chosen a linear extension.  

Our linear extension gives a total order to each subset $S_j$ of $S$.
With respect to this order, $j\in E_k$ if and only if $\min S_j =k$.
Clearly if $i\sim_S j$, the sets $S_i=S_j$ have the same minimum, so
each $E_k$ is a union of one or more
blocks of $\sim_S$.  Since there are only 
$r=\abs{S}$ blocks, though, 
it follows each $E_k$ is equal to exactly one of them.
\end{proof}
In the other direction, 
\begin{lemma}[Prop.~2.8(3), \cite{FK04}]\label{lem:chains}
For any chain of flats $F_1<F_2<\cdots<F_d=[n]$ in 
an intersection lattice $L(\M)$, 
for any building set, there exists a nested set $S\in\Ne(G)$ and a
linear extension $S=\set{X_1,\ldots,X_d}$ for which
$F_k=\bigvee_{i=1}^k X_i$, for all $1\leq k\leq d$.
\end{lemma}

\subsection{Wonderful models II}\label{ss:wonderful2}
Once again, our compactification took place in a projective toric variety.
Let $X_{\dCP}(\M,G)$ denote the closure of $\PT^{n-1}$, embedded diagonally
in the product of projective spaces in \eqref{eq:diagonalmap}.  (Our
notation reflects the fact that the toric variety depends only on the 
matroid of $\A$.)

For each flat $X\in L(\A)$, let
\begin{equation}\label{eq:simplex}
\Delta_X=\conv\set{e_i\colon i\in X}
\end{equation}
denote the weight polytope of $\P^{\abs{X}-1}$.  The weight polytope of
$X_{\dCP}(\M,G)$ is the Minkowski sum $\sum_{X\in G}\Delta_X$, by
Proposition~\ref{prop:minkowski} (see \cite[\S 6]{FS05}).
It should also be mentioned that Ardila, Benedetti and Doker~\cite{ABD10}
have found some remarkable
results relating matroid polytopes and Minkowski sums of simplices,
building on Postnikov's work on generalized permutohedra~\cite{Po09}.
In particular, they obtain a combinatorial
formula for the degree of the toric variety $X_{\M}$.

In Section~\S\ref{ss:vc}, we saw that we could replace our ambient,
complete toric variety by a minimal one, indexed by the Bergman fan.
Similarly, we can describe the minimal toric varieties that give the
wonderful models, using a construction that first appeared in
\cite{FY04}.

We continue to assume that $\M$ is connected.
Let $G\subseteq L(\M)$ be a building set.  
We construct a rational, polyhedral
fan $\Nfan(\M,G)$ in $N_\R$ by taking
the cone over the geometric realization of the complex $\Ne_0(G)$.  That is, 
for each nested set $S\subseteq G$, define a cone
\begin{equation}\label{eq:NScone}
\sigma_S=\R_{\geq0}\,\conv\set{e_X\colon X\in S},
\end{equation}
and let $\Nfan(\M,G)=\set{\sigma_S\colon S\in\Ne(G)}$.
Since $\Ne(G)$ is a simplicial complex, $\Nfan(\M,G)$ is a simplicial fan.
(From Exercise~\ref{ex:NScone}, we would have constructed the same fan
by taking the cone over $\Ne(G)$ instead, since $e_{[n]}=0$ in the 
lattice $N$.) 

\begin{proposition}\label{prop:minimal_dCP}
For any arrangement $\A$ and building set $G$, the toric variety
$X_{\Nfan(\M,G)}$ is a subvariety of $X_{\dCP}(\M,G)$.
\end{proposition}
\begin{proof}
An inclusion of normal toric varieties is given by an inclusion of 
toric fans, so 
let $\Sigma(G)$ denote the normal fan of $\sum_{X\in G}\Delta_X$.
A straightforward argument with the weight polytope shows that
$X_{\dCP}(\M,G)$ is normal, so $\Sigma(G)$ is its toric fan.  It is
enough to check that each cone in the
fan $\Nfan(\M,G)$ is contained in some cone of $\Sigma(G)$.

The normal fan of a Minkowski
sum of polytopes is the coarsest common refinement of their respective
normal fans: that is, two functionals $u,v\in N_{\R}$ lie in the same
open cone of $\Sigma(G)$ if and only if they achieve their minimum on the
same faces of $\Delta_X$, for each $X\in G$.
Faces of a standard simplex are indexed by subsets of coordinates, so
for a given $u\in N_{\R}$ and subset $I\subseteq[n]$, let
\[
\min\nolimits_I(u)=\set{i\in I\colon u_i=m,
\text{~where $m=\min\set{u_j}_{j\in I}$}}.
\]
Then, translating the above, $u\sim v$ in $\Sigma(G)$ if and only if
$\min_X(u)=\min_X(v)$ for all $X\in G$.

Now suppose that $u$ lies in the interior of a cone of $\Nfan(\M,G)$ indexed
by a nested set $S\in \Ne(G)$, so $u=\sum_{X\in S}c_X e_X$ for some
coefficients $c_X>0$.  We need to show that the set $\min_X(u)$ is
independent of the coefficients, for each $X\in G$.  For this, let
$S_i=\set{X\in S\colon i\in X}$, for each $i\in [n]$.  Since
$u_i=\sum_{X\in S_i}c_X$, we have $u_i=u_j$ if $S_i=S_j$, and $u_i<u_j$
if $S_i\subset S_j$, for all $i$, $j$.

By Lemma~\ref{lem:minimal}, the set $\set{S_i\colon i\in X}$ has a minimal
element $S_{i_0}$, for some $i_0\in X$.  But then
$\min_X(u)=\set{i\in X\colon S_i=S_{i_0}}$, which depends only on the
nested set $S$, so we are done.
\end{proof}

\begin{theorem}\label{thm:FY}
For any arrangement $\A$ and building set $G$,
\[
Y_{\dCP}(\A,G)=\overline{\P U(\A)}\subseteq X_{\Nfan(\M,G)}.
\]
Moreover, the toric variety $X_{\Nfan(\M,G)}$ is minimal, in the sense
that $\overline{\P U(\A)}$ intersects every open torus orbit.
\end{theorem}

Once again, the compactifying ambient space depends only on the matroid.  
The situation is parallel to the one
with $Y_{\vc}(\A)$, since our small toric variety
is not complete.  However, Feichtner and Yuzvinsky~\cite{FY04} show that
the cones of $\Nfan(\M,G)$ 
are unimodular, so the toric variety $X_{\Nfan(\M,G)}$ is
always smooth, in contrast to $X_{\Be(\M)}$.

This construction also relates nicely to De Concini and Procesi's description
of $Y_{\dCP}(\A,G)$ as an iterated blowup, as shown in \cite[\S 6]{FY04}.
Roughly speaking, one builds the fan $\Nfan(\M,G)$ by starting with
the fan for $\P^{n-1}$ (Example~\ref{ex:Pn}), noting that $\Nfan(\M,G)$
always contains the rays through $e_i$, for $1\leq i\leq n$.  Then
one adds rays  $e_S$ and subdivides cones,
for the remaining elements $S\in G$, in non-increasing order
with respect to $L(\M)$.  Stellar subdivision in a fan corresponds to
blowing up closed orbits in toric varieties: see, e.g., \cite[\S 3.3]{CLS11}.
By blowing up $\P^{n-1}$ along
proper transforms of coordinate subspaces, one obtains a complete toric
variety; then $X_{\Nfan(\M,G)}$ is the
subvariety obtained by deleting the orbits not indexed by nested sets.

Our constructions
so far are summarized in Table~\ref{tab:cpct}, where $V$ is a linear space,
$\A=\A(V)$, and $\M=\M(V)$.
\begin{table}\label{tab:cpct}
\caption{Arrangement compactifications in \S\ref{sec:cpct}}
\[
\begin{array}{llccc}
\toprule
\multicolumn{2}{c}{\text{Compactification}} & \text{complete toric variety} & \text{weight polytope} & \text{reference} \\ \midrule
\P V &  \text{tautological}     & \P^{n-1} & \Delta^{n-1} & \\
Y(\A) & \text{reciprocal plane} & \P^{n-1} & -\Delta^{n-1} & \eqref{eq:recip}\\
Y_{\vc}(\A) & \text{visible contours} & X_{\M}^{-1}\subseteq \Gr_{d,n}
 & -P_{\M}(\cong
P_{\M^*}) & \eqref{eq:Xvc}\\
Y_{\dCP}(\A,G) & \text{wonderful model} & X_{\dCP}(\M,G)
\subseteq\prod\limits_{X\in G}\P^{\abs{X}-1} % include?
 &
\sum_{X\in G}\Delta_X & \eqref{eq:XdCP} \\ \bottomrule
\end{array}
\]
\vskip10pt
\[
\begin{array}{lclc}
\toprule
\text{Compactification} & \text{minimal toric variety}  & \text{toric fan} 
&  \text{reference} \\ \midrule
Y_{\vc}(\A) & X_{\Be(\M)}\subseteq X_{\M}^{-1}
&\text{Bergman fan} & \text{Theorem~\ref{thm:bergman}}\\
Y_{\dCP}(\A,G) & X_{\Nfan(\M,G)}\subseteq X_{\dCP}(\M,G)
&\text{nested set fan}& \text{Theorem~\ref{thm:FY}}\\ \bottomrule
\end{array}
\]

\end{table}

\subsection{Comparisons}\label{ss:comparison}
For a fixed arrangement, we now have a number of compactifications.  In
particular, if $G_1\subseteq G_2$ are two building sets for an intersection
lattice $L(\M)$, then the obvious projection map
\[
\prod_{X\in G_2}\P V_X\twoheadrightarrow \prod_{X\in G_1}\P V_X
\]
induces a map of wonderful compactifications $Y_{\dCP}(\A,G_2)\to
Y_{\dCP}(\A,G_1)$.  From De Concini and Procesi's original point of view,
this map blows down the boundary divisor components indexed by $G_2-G_1$.

From the toric point of view, Feichtner and M\"uller prove in
\cite[Thm.~4.2]{FM05} that the fan $\Nfan(\M,G_2)$ is a refinement of the fan
$\Nfan(\M,G_1)$.  More precisely, the former fan is
obtained by stellar subdivisions of that latter.  Geometrically,
the map of toric varieties
\[
X_{\Nfan(\M,G_2)}\twoheadrightarrow X_{\Nfan(\M,G_1)}
\]
blows down the codimension-$1$ torus orbits indexed by $G_2-G_1$.
It follows that the support of the fan $\Nfan(\M,G)$ is independent of $G$:
that is, 
\begin{equation}\label{eq:supports}
\abs{\Nfan(\M,G_{\min})}=\abs{\Nfan(\M,G)}=
\abs{\Nfan(\M,G_{\max})}
\end{equation}
for any building set $G$.

A key discovery in \cite{FS05} is the following (and we outline a proof below).
\begin{theorem}\label{thm:refines}
For any matroid $\M$, the nested set fan $\Nfan(\M,G_{\min})$ 
refines the Bergman fan $\Be(\M)$.
\end{theorem}
This means that there is a natural map of toric varieties,
\[
X_{\Nfan(\M,G_{\min})}\twoheadrightarrow X_{\Be(\M)},
\]
and a corresponding map of compactifications $Y_{\dCP}(\A,G)
\twoheadrightarrow Y_{\vc}(\A)$
for any building set $G$ (by factoring through $G_{\min}=L_\irr(\M)$).

It turns out that, if for every
flat $X\in L_\irr(\M)$ it happens that $\M/X$ is also connected,
then this subdivision
can be accomplished without adding new vertices.  If an even stronger
condition holds, the two fans are actually equal:
\begin{theorem}[Thm.~5.3,~\cite{FS05}]\label{thm:comparison}
The fans $\Nfan(\M,G_{\min})$ and $\Be(\M)$ are equal if and only if
the matroid $(\M|Y)/X$ is connected for all pairs of flats $X\leq Y$,
where $Y\in G_{\min}$.
\end{theorem}
Accordingly, if this matroid condition is satisfied, then the visible
contours compactification $Y_{\vc}(\A)$ is equal to the wonderful
compactification given by the minimal building set.
Before continuing with examples, we give another argument that makes use
of nested set combinatorics.
\begin{proof}[Proof of Theorem~\ref{thm:refines}]
For any matroid $\M$, we first show that 
each cone of $\Nfan(\M,G_{\min})$ is contained
in a cone of $\Be(\M)$. 
It is enough to verify this for a cone $\sigma_S$ (defined in
\eqref{eq:NScone}) for a maximal nested set $S$, so suppose
$u=\sum_{X\in S}c_X e_X$, where each coefficient $c_X>0$.  Then $u_i=u_j$
if $i\sim_S j$, using the equivalence relation from
\S\ref{ss:wonderful}.  Let $\F(-u)=(F_0,\ldots,F_k)$ be the chain of
subsets \eqref{eq:Fu}: then each set $F_i-F_{i-1}$ is a union of 
$\sim_S$-blocks.  Let us temporarily assume that
\begin{equation*}\tag{$\star$}\label{eq:condition}
u_i=u_j\text{~if and only if~}i\sim_S j:
\end{equation*}
i.e., each $F_i-F_{i-1}$ is a single block, 
and impose a total order on the nested set $S$ as follows.  
For flats $X,Y\in S$,
recall $X=\min S_i$ and $Y=\min S_j$ for some $i$, $j$, by
Lemma~\ref{lem:Smin}.  In that case, put $X\prec Y$
if $u_i> u_j$.  This is a linear extension of $S$, since if we
had $X\leq Y$ in $L(\M)$, then $Y\in S_i$ as well, whence $S_j\subseteq S_i$,
and $u_j\geq u_i$.  Write $S=\set{X_1,\ldots,X_d}$, numbering the
flats in $\prec$-order.
By our assumption that each set $F_i-F_{i-1}$ is a block of $\sim_S$, 
for $1\leq i\leq k$, 
by Lemma~\ref{lem:partition}, we see $F_i=\bigvee_{j=1}^i X_i$, for 
each $i$.

This is a chain of flats of $\M$, so by Theorem~\ref{thm:flacet0}, 
$-u$ achieves its minimum on a face $F=\M_{-u}$ indexed by a matroid
without loops, so $u\in\abs{\Be(\M)}$.  
Moreover, the chain $\F(-u)$ and face $F$ depended only on the set $S$
and not on the choice of $u$, provided \eqref{eq:condition} was 
satisfied.  By convexity, the same is true for all $u$ in the interior
of $\sigma_S$, so this cone lies in a cone of $\Be(\M)$.

It follows that $\abs{{\Nfan(\M,G)}}\subseteq \abs{{\Be(\M)}}$.
To show that the supports are equal, consider a vector $u\in\abs{\Be(\M)}$.
By Theorem~\ref{thm:flacet0}, elements of the chain $\F(-u)$ are flats of
$\M$.  By Lemma~\ref{lem:chains},
there exists a linear extension of a nested set $S=\set{X_1,\ldots,X_d}$ 
for the flats in the chain are $\set{\bigvee_{i=1}^k X_i}_{1\leq k\leq d}$.
By Lemma~\ref{lem:partition} and the first part of this proof, $u\in \sigma_S$.
\end{proof}

\begin{example}\label{ex:delA3_compare}
Continuing our usual Example~\ref{ex:delA3_nested}, we see the
nested set complex has one more vertex than the Bergman complex.
Figure~\ref{fig:refined} compares the Bergman fan with the nested
set complexes for the minimal and maximal building sets,
$G_{\min}=L_\irr(\M)$ and $G_{\max}=L_{\geq1}(\M)$.

The toric variety for \ref{fig:irr_refinement} is obtained by blowing up
$\P^4$ along two coordinate lines, then deleting the torus orbits of dimension
$\leq1$, as well as the $2$-dimensional orbits indexed by non-nested pairs
of rays.  The closure of $\P U(\A)$ inside is the blowup of $\P^2$ at
the two triple intersections.

Then $Y_{\vc}(\A)$ is obtained from $Y_{\dCP}(\A,G_{\min})$ by blowing down
the boundary component corresponding to the line through the
two triple points, giving $\P^1\times\P^1$.
\end{example}
\begin{figure}
\subfigure[$B_{\M}$]{
\begin{tikzpicture}[scale=1.65]
\tikzstyle{every node}=[circle,draw,fill=white,inner sep=1.5pt]
\node[label=above:$135$] (A135) at (0,1.7) {};
\node[label=left:$5$] (A5) at (-1,0) {}; 
\node[label=right:$4$] (A4) at (1,0) {};
\node[label=left:$2$] (A2) at (-0.25,0.8) {};
\node[label=right:$3$] (A3) at (0.25,0.8) {};
\node[label=right:$\;124$] (A124) at (0,0.5) {};
\node[label=below:$\hat{1}$] (A1) at (-0.04,0.32) {};
% order matters because of lines crossing
\draw[style=thin,dashed,color=gray] (A2) -- (A1) -- (A3);  -- ghost vertex
\draw[style=thin,dashed,color=gray] (A4) -- (A1) -- (A5);

\draw[line width=4pt,color=white] (A124) -- (A5);
\draw[line width=4pt,color=white] (A4) -- (A124);

\draw[style=thin,color=gray] (A3) -- (A124) -- (A5);
\draw[style=very thick,color=red!40!black] (A2) -- (A3);
\draw[line width=6pt,color=white] (A135) -- (A124);
\draw[style=very thick,color=red!40!black] (A5) -- (A135) -- (A124) -- (A4) -- (A5) -- (A2) -- (A124);
\draw[style=thin,color=gray] (A4) -- (A135) -- (A2);

\draw[style=very thick,color=red!40!black] (A135) -- (A3) -- (A4);

\end{tikzpicture}
}
%\quad
\subfigure[$\Sigma(\A,G_{\min})$]{
\begin{tikzpicture}[scale=1.65]
\tikzstyle{every node}=[circle,draw,fill=white,inner sep=1.5pt]
\node[label=above:$135$] (A135) at (0,1.7) {};
\node[label=left:$5$] (A5) at (-1,0) {}; 
\node[label=right:$4$] (A4) at (1,0) {};
\node[label=left:$2$] (A2) at (-0.25,0.8) {};
\node[label=right:$3$] (A3) at (0.25,0.8) {};
\node[label=left:$\;124$] (A124) at (0,0.5) {};
%\node[label=below:$\hat{1}$] (A1) at (-0.04,0.32) {};
\node[label={[label distance=1em]150:$1$}] (A1) at (0,1.1) {};
% order matters because of lines crossing
%\draw[style=thin,dashed,color=gray] (A2) -- (A1) -- (A3);  -- ghost vertex
%\draw[style=thin,dashed,color=gray] (A4) -- (A1) -- (A5);

\draw[line width=4pt,color=white] (A4) -- (A124);

\draw[style=thin,color=gray] (A3) -- (A124) -- (A5);
\draw[style=very thick,color=red!40!black] (A2) -- (A3);
\draw[line width=6pt,color=white] (A135) -- (A1) -- (A124);
\draw[style=very thick,color=red!40!black] (A5) -- (A135) -- (A1) -- (A124) -- (A4) -- (A5) -- (A2) -- (A124);
\draw[style=thin,color=gray] (A4) -- (A135) -- (A2);

\draw[style=very thick,color=red!40!black] (A135) -- (A3) -- (A4);

\end{tikzpicture}\label{fig:irr_refinement}
}
%\quad
\subfigure[$\Sigma(\A,G_{\max})$]{
\begin{tikzpicture}[scale=1.65]
\tikzstyle{every node}=[circle,draw,fill=white,inner sep=1.5pt]
\node[label=above:$135$] (A135) at (0,1.7) {};
\node[label=left:$5$] (A5) at (-1,0) {}; 
\node[label=right:$4$] (A4) at (1,0) {};
\node[label=left:$2$] (A2) at (-0.25,0.8) {};
\node[label=right:$3$] (A3) at (0.25,0.8) {};
\node[label=left:$\;124$] (A124) at (0,0.5) {};
\node[label=above:$45$] (A45) at (0,0) {};
\node[label=right:$34$] (A34) at (0.625,0.4) {};
\node[label={[label distance=0.2em]150:$25$}] (A25) at (-0.625,0.4) {};
\node (A23) at (0,0.8) {};
\coordinate[label=right:$23$] (lab23) at (0.8,1.2);
\node[label={[label distance=1em]150:$1$}] (A1) at (0,1.1) {};
%\node[label=below:$\hat{1}$] (A1) at (-0.04,0.32) {};
% order matters because of lines crossing
%\draw[style=thin,dashed,color=gray] (A2) -- (A1) -- (A3);  -- ghost vertex
%\draw[style=thin,dashed,color=gray] (A4) -- (A1) -- (A5);

\draw[line width=4pt,color=white] (A4) -- (A124);

\draw[style=thin,color=gray] (A3) -- (A124) -- (A5);
\draw[style=very thick,color=red!40!black] (A2) -- (A23) -- (A3);
\draw[line width=6pt,color=white] (A135) -- (A1) -- (A124);
\draw[style=very thick,color=red!40!black] (A5) -- (A135) -- (A1) -- (A124) -- (A4) -- (A45) -- (A5) -- (A25) -- (A2) -- (A124);
\draw[style=thin,color=gray] (A4) -- (A135) -- (A2);

\draw[style=very thick,color=red!40!black] (A135) -- (A3) -- (A34) -- (A4);
%\draw
%[decoration={markings,mark=at position 1 with {\arrow[scale=2]{>}}},
%postaction={decorate}]
\draw[->] (lab23) edge [bend right=45] (A23);

\end{tikzpicture}
}
\caption{refinements of $B_{\M}$, Example~\ref{ex:delA3_compare}}
\label{fig:refined}
\end{figure}
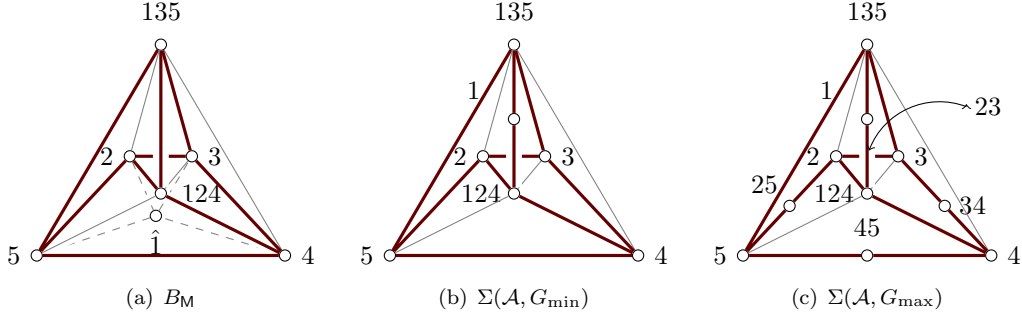
\begin{example}
The braid arrangements, Example~\ref{ex:braid}, are particularly 
interesting.  The original construction of $Y_{\dCP}(\A,G_{\min})$ is the
case $X=\C$ of a configuration space compactification of $F(X,d+1)$ due to 
Fulton and MacPherson~\cite{FM94}.
Points in the arrangement complement for
$V(A_d)$ are also in bijection with genus zero curves with $d+2$
marked points, and $Y_{\dCP}(\A,G_{\min})$
agrees with the Deligne-Knudsen-Mumford compactification.

The visible contours and wonderful compactifications
agree (for $G=L_{\irr}$), which we can check as follows.
If $\pi_X$ is a partition of the
set $[d+1]$, refinements $\pi_Y$ of $\pi_X$ are in bijection with
partitions of the set of blocks of $\pi_X$.  With this in mind,
suppose we have flats $X< Y$ and $Y$ is irreducible.  As a partition
of the blocks of $\pi_X$, we see $\pi_Y$ still has only one block of
size greater than one, so the interval $[X,Y]$ is order-isomorphic to
the intersection lattice of another braid arrangement (of rank $r(Y)-r(X)$).  
Accordingly, $(\M|Y)/X$ is connected, and the claim follows by applying
Theorem~\ref{thm:comparison}.
\end{example}
\begin{example}\label{ex:unif_bergman}
If $\M(V)=U_{d,n}$, the uniform matroid, the situation is straightforward.
Assume $d<n$; then $G:=L_\irr(U_{d,n})=\set{\set{i}\colon 1\leq i\leq n}\cup
\set{[n]}$.  The nested sets are the subsets of size at most $d$, so
the maximal cones in $\Nfan(U_{d,n},G)$ are just the coordinate cones
spanned by $d-1$-element subsets of $\set{e_i\colon 1\leq i\leq n}$.
This refines the Bergman fan, but no strictly coarser fan with the
same support is possible, so $\Be(U_{d,n})=\Nfan(U_{d,n},G)$, as we
see for $U_{2,4}$ in Figure~\ref{fig:octahedron_b}.
Our ambient toric variety
is $\P^{n-1}$, minus all coordinate subspaces of codimension $d$ and
higher, and $Y_{\vc}(\A)=Y_{\dCP}(\A,G)=\P^{d-1}$.
\end{example}

\section{Concluding remarks}
By limiting our discussion to linear spaces in tori (that is, hyperplane
complements), this tutorial stops short of ``modern'' tropical geometry.
In particular, Tevelev's paper \cite{Te07} broadly
generalizes the relationship we saw here between the Gel$'$fand-MacPherson
construction, the Bergman fan, and the visible contours compactification.

We have also neglected any discussion of intersection theory here,
due to the constraints of time and space, although this is also an interesting
part of the story above.  We leave the last exercise as a
possible starting point for further reading in this direction.

\begin{exercise}
We noted in Remark~\ref{rem:degree} that the support of the
Bergman fan is a cone over
a wedge of $\mu$ spheres of dimension $d-2$.  From the discussion in
\S\ref{ss:comparison}, the same is true of the nested set fans $\Nfan(\M,G)$.

It is also known that the number $\mu$ is the degree of
the reciprocal plane $Y(\A)$, from \cite[Lemma~2]{PS06}: that is,
$\mu$ is the number of points at which $Y(\A)$ intersects a generic
projective linear subspace $W$ in $\P^{n-1}$ of dimension $n-d$.

A basic idea of tropical intersection theory is that, under suitable
conditions, one can compute an intersection number by intersecting
tropicalizations: see, e.g., \cite{Ka12} for details.
Using Theorem~\ref{thm:tropical}, we know the tropicalizations of
$Y(\A)$ and $W$ are given by $\abs{\Be_{\M(V)}}$ and $-\abs{\Be_{\M(W)}}$,
respectively.  The linear space $W$ realizes the uniform matroid, so
its Bergman fan was computed in Example~\ref{ex:unif_bergman}.

Use the tropical intersection product (also known as the fan displacement
rule) to give a ``tropical'' proof that $\mu$ is the degree of $Y(\A)$.
\end{exercise}

\begin{ack}
The author is indebted to a number of people for helpful discussions about this 
material, in particular: Maria Angelica Cueto, Mike Falk, Eva Feichtner, 
June Huh, Eric Katz, Diane MacLagan, Kristin Shaw, and Jenia Tevelev.
\end{ack}

%\nocite{*}
\bibliographystyle{amsalpha}
\renewcommand{\MR}[1]
{\href{http://www.ams.org/mathscinet-getitem?mr=#1}{MR#1}}
%\bibliography{cpct}
%\end{document}
\newcommand{\arxiv}[1]
{\texttt{\href{http://arxiv.org/abs/#1}{arXiv:#1}}}

\newcommand{\etalchar}[1]{$^{#1}$}
\def\cprime{$'$}
\providecommand{\bysame}{\leavevmode\hbox to3em{\hrulefill}\thinspace}
%\providecommand{\MR}{\relax\ifhmode\unskip\space\fi MR }
\providecommand{\MRhref}[2]{%
  \href{http://www.ams.org/mathscinet-getitem?mr=#1}{#2}
}
\providecommand{\href}[2]{#2}

\end{document}